\documentclass[11pt]{article}
\usepackage[latin1]{inputenc}
\usepackage[T1]{fontenc}
\usepackage[style=numeric]{biblatex}
\usepackage{amsmath}
\usepackage{amsfonts}
\usepackage{amssymb}
\usepackage{amsthm}
\usepackage[affil-it]{authblk}
\usepackage{enumitem}
\usepackage[normalem]{ulem}
\usepackage{dsfont}
\usepackage{float}
\usepackage{bbm}
\usepackage{mathtools}

\usepackage[colorlinks,linkcolor=blue,citecolor=blue,urlcolor=blue]{hyperref}
\usepackage[capitalise]{cleveref}
\usepackage[left=3.5cm, right=3.5cm, bottom=3cm, top=3cm]{geometry}

\bibliography{reference.bib}

\newtheorem{theorem}{Theorem}[section]
\newtheorem{lemma}[theorem]{Lemma}
\newtheorem{proposition}[theorem]{Proposition}

\newtheorem{assumption}[theorem]{Assumption}
\theoremstyle{definition}
\newtheorem{defn}[theorem]{Definition}
\theoremstyle{remark}
\newtheorem*{remark}{Remark}

\numberwithin{equation}{section}

\crefname{example}{Example}{Examples}

\newcommand{\R}{\ensuremath{\mathbb{R}}}

\newcommand{\norm}[1]{\ensuremath\left\|#1\right\|}

\renewcommand{\geq}{\geqslant}
\renewcommand{\leq}{\leqslant}

\providecommand{\keywords}[1]
{
  \textbf{Keywords---} #1
}

\begin{document}
\title{An SPDE with Robin-type boundary for a system of elastically killed diffusions on the positive half-line}
\author{Ben Hambly\footnote{Mathematical Institute, University of Oxford, Email: hambly@maths.ox.ac.uk} , Julian Meier\footnote{Mathematical Institute, University of Oxford, Email: julian.meier@maths.ox.ac.uk (corresponding author)} , and Andreas S{\o}jmark\footnote{Department of Statistics, London School of Economics, Email: a.sojmark@lse.ac.uk}} 
\date{\today} 

\maketitle

\begin{abstract}
 We consider a system of particles undergoing correlated diffusion with elastic boundary conditions on the half-line. By taking the large particle limit we establish existence and uniqueness for the limiting empirical measure valued process for the surviving particles. This process can be viewed as the weak form for an SPDE with a noisy Robin boundary condition satisfied by the particle density. We establish results on the $L^2$-regularity properties of this density process, showing that it is well behaved in the interior of the domain but may exhibit singularities on the boundary at a dense set of times.
We also show existence of limit points for the empirical measure in the non-linear case where the particles have a measure dependent drift. We make connections for our linear problem to the corresponding absorbing and reflecting SPDEs, as the elastic parameter takes its extreme values.  
\end{abstract}

\keywords{Particle System, Common Noise, Mean-field type SPDE, Elastic killing, Robin boundary}

\section{Introduction}

Consider a conditionally i.i.d.~system of $N $ reflected It{\^o} diffusions $\{X^i\}_{1\leq i \leq N}$ living on the positive half-line $[0,\infty)$ with a given finite time horizon $T$. The 
initial values $\{X_0^{i}\}_{1 \leq i \leq N}$ are chosen independently from 
a common distribution $\nu_0$ on $(0,\infty)$, and each $X^i$ has diffusive dynamics
\begin{align} \label{particle SDE}
\begin{split}
    X_t^{i} &= X_0^i + \int_0^t \mu(s,X_s^{i})ds + \int_0^t \sigma(s,X_s^{i}) \rho(s) dW_s^0 \\
    &+ \int_0^t \sigma(s,X_s^{i})(1- \rho(s)^2)^{\frac{1}{2}}dW_s^i+L_t^i, \quad i=1,\ldots, n,
\end{split}
\end{align}
where $L^i$ is the local time of $X^i$ at the origin and $W^0,W^1,\ldots,W^N$ are independent Brownian motions. The precise assumptions on the coefficients are left for Section 2.

We shall also refer to each $X^i$ as the $i$'th particle. On top of the dynamics \eqref{particle SDE}, we wish to consider the `killing' of each particle after a suitable (random) amount of time has been spent at the boundary. To this end, we fix a parameter $\kappa >0$ and define
\begin{equation}\label{eq:default|_time}
\tau^i \coloneqq \inf\{t>0 : L_t^i > \chi^i \},\quad i=1,\ldots,n,
\end{equation}
for a family $(\chi^i)_{1 \leq i \leq N}$ of i.i.d.~exponential random variables with rate $\kappa>0$. This captures the idea of each particle being killed elastically at the origin with parameter $\kappa >0$, and we say $\tau^i$ is the elastic killing time of $X^i$. Note that the measure $dL^i_t$, induced by the local time, is supported on the set $\{t:X_t^i = 0 \}$, so we know that $X_{\tau^i}^i=0$, meaning that each particle can only be killed when it is at the boundary.
 
The main object of study in this paper is the family of empirical measures for the surviving particles, for every time $t$ and any number of particles $N$,
given by

\begin{equation}\label{eq:empirical_measures}
\nu_t^N(dx) := \frac{1}{N} \sum_{i=1}^N \delta_{X_t^i}(dx) \mathbbm{1}_{t<\tau^i},\quad t\in[0,T],\quad N\geq 1.
\end{equation}
A similar particle system with elastic killing but without a common noise has been studied in \cite{SojmarkWorkingPaper22} to model an epidemic advancing through a population.

In financial applications, $\tau^i$ could represent the \emph{default time} of an entity whose \emph{financial health} is modelled by $X^i$. We then say that an entity is in \emph{distress} when $X^i_t$ is zero. In practice, an entity could for example be an asset such as a credit instrument forming part of a larger credit default obligation or a defaultable loan in a mortgage backed security. It could also represent an entire company or a financial institution. Either way, if enough time is spent in this distressed state, the inevitable will happen: the clock $\tau^i$ rings and the $i$'th entity is declared to be in default. Due to the structure of \eqref{eq:default|_time}, where the underlying exponential random variables are not observed, we obtain the desirable properties that the default times are not predictable and their laws are memoryless.

In order to mange the risk of investments and regulate or monitor a financial system, it becomes important to understand, respectively, the pricing of credit derivatives on portfolios of a large number of defaultable entities or the computation of risk measures related to the overall health of a large number of financial institutions. Both cases will typically involve the computation of expected values $\mathbb{E}[F(\nu^N)]$ for some functional $F$ of the paths $t\mapsto \nu_t^N$ up until a terminal time $T$. Viewing $\nu^N:=(\nu_t^N)_{t\geq 0}$ as a measure-valued c\`adl\`ag stochastic process, it is then natural to look for a functional limit theorem as $N\rightarrow \infty$ and thereby seek to identify a good approximation $\mathbb{E}[F(\nu)]$, where the limit $\nu=(\nu_t)_{t\geq 0}$ satisfies an evolution equation driven only by the common factor $W^0$.

Closely related to our setting, \cite{HamblyetAl201} provides such a result motivated by the pricing of credit default swap indices, which involves computing the expected value of suitable functionals of the loss process 
$t\mapsto \mathcal{L}^N_t := 1- \nu^N_t(0,\infty)$, 
giving the proportion of defaults.
However, \cite{HamblyetAl201} does not work with elastic default times for reflected dynamics as we do here. Instead, they consider a constant coefficient version of the system \eqref{particle SDE} without reflection, where defaults are declared to happen immediately at the first time the financial health $X^i$ hits zero. The paper \cite{Ledger2014} studies precise regularity properties of solutions to the resulting SPDE with absorbing boundary satisfied by $\nu$ (in the limit as $N\rightarrow \infty$). Moreover, \cite{Giles_Reisinger, Wang} have proposed and analysed efficient numerical schemes for simulating this SPDE, based on finite difference schemes in combination with multi-level or multi-index Monte Carlo methods, which can yield substantial gains over the simulation of the finite system for reasonable sizes of $N$. Finally, \cite{HamblyLedger2017, HamblySojmark2019} have considered the well-posedness of broader classes of such parabolic SPDEs with absorbing boundary and \cite{AhmadHamblyLedger2018} considers an SPDE with an absorbing boundary on a compact interval which is applied to the pricing of Mortgage Backed Securities.

In a class of sufficiently regular solutions, uniqueness of the SPDE with absorbing boundary in \cite{HamblyetAl201} can be deduced from the general Sobolev theory on Dirichlet problems for SPDEs developed in \cite{Krylov1994, Krylov1998}. As far as we are aware, no such theory exists for the type of SPDEs with reflecting or elastic boundaries that we consider in the present paper. For details on the specific type, see \eqref{eq:classical_SPDE}-\eqref{eq:classical_SPDE_BC} below with $\kappa=0$ or $\kappa>0$, respectively.

In \cite{HamblyetAl201}, adapting the approach of \cite{Kotelenez, Kurtz_Xiong} for the whole space to a half-line with absorbing boundary, uniqueness of a priori measure-valued solutions is proved by means of uniform $L^2$ energy estimates for suitable mollifications of the solutions. This method only works because the authors can quantify the second moment of the mass near the boundary as having decay of at least order $\mathbb{E}[\int_0^T|\nu_t(0,\varepsilon)|^2dt]=O(\varepsilon^{3+\gamma})$, as $\varepsilon \downarrow 0$,  for some $\gamma>0$. Even with constant coefficients except for a time dependent correlation $t\mapsto \rho(t)$, it becomes unclear how to obtain such a strong order of decay. In \cite{HamblyLedger2017}, it was instead realised that, by working in the dual, $H^{-1}$, of the first Sobolev space, $H^1$, uniqueness in the case of an absorbing boundary can be obtained from energy estimates in $H^{-1}$ as soon as we know that the first moment of the mass near the boundary satisfies $\mathbb{E}[\int_0^T|\nu_t(0,\varepsilon)|dt]=O(\varepsilon^{1+\gamma})$, as $\varepsilon \downarrow 0$,  for some $\gamma>0$, which is very easy to verify for the limit points of the empirical measures in the absorbing case.

The aforementioned first moment control, however, is much too strong a requirement for elastic and reflecting boundaries. Indeed, the absolute best one could hope for is that $\mathbb{E}[\int_0^T|\nu_t(0,\varepsilon)|dt]$ is of order $O(\varepsilon)$ as $\varepsilon \downarrow 0$. Fortunately, returning instead to work with the second moment, we are able to implement the $H^{-1}$ approach from \cite{HamblyLedger2017} by showing that limit points of \eqref{eq:empirical_measures} must satisfy $\mathbb{E}[\int_0^T|\nu_t(0,\varepsilon)|^2dt]=O(\varepsilon^{1+\gamma})$, as $\varepsilon \downarrow 0$,  for some $\gamma>0$. This is enough for us to succeed in establishing uniqueness within a broad class of measure-valued solutions that includes the limit points of the empirical measures.

\subsection{Summary of main results}

By analogy with the connection between an elastically killed Brownian motion and the heat equation with a convective Robin boundary, one heuristically expects limit points of \eqref{eq:empirical_measures} to take the form of a stochastic process $(V_t)_{t\geq0}$, where each $V_t$ is a random sub-probability density on $[0,\infty)$ governed, in a suitable sense, by the evolution equation
\begin{equation}\label{eq:classical_SPDE}
dV_t =   \bigl(  \partial_{xx} \bigl( \frac{\sigma^2_t}{2} V_t\bigr)  - \partial_x ( \mu_t  V_t)  \bigr)dt  - \rho_t  \partial_x \left( \sigma_t V_t \right) dW^0_t,
\end{equation}
in the interior, together with the noisy Robin-type boundary condition
\begin{equation}\label{eq:classical_SPDE_BC}
\partial_x \bigl(\frac{\sigma^2(\cdot,t)}{2}  V_t \bigr)(0) = \kappa \frac{\sigma^2(t,0)}{2}V_t(0) + \mu(t,0) V_t(0) +\rho(t)\sigma(t,0)V_t(0) \Dot{W}^0_t.
\end{equation}
Note that, to simplify notation, we will often suppress the spatial dependence and simply write $\mu_t = \mu(t,x)$ and $\sigma_t = \sigma(t,x)$ as well as $\rho_t = \rho(t)$.


Our main contributions in this paper are as follows. Firstly, we identify a suitable  weak formulation of the SPDE \eqref{eq:classical_SPDE}-\eqref{eq:classical_SPDE_BC}, which is shown to uniquely characterize a functional limit theorem for the empirical measures \eqref{eq:empirical_measures} as $N\rightarrow \infty$ (Theorems \ref{existenceThm} and \ref{uniqueness}). Secondly, we show that the unique solution $\nu$ of our weak formulation has a density $V_t$ at all times, that this density is square-integrable for almost all times, and that it is square-integrable away from the boundary at all times (Theorem \ref{thm:L2-density} and Proposition \ref{prop:Burdzy}). Thirdly, we connect the study of absorbing, reflecting, and elastic boundaries within the class of linear parabolic SPDEs on a half-line analysed here. We do this by showing that well-posedness with a reflecting boundary is also captured by our arguments (Theorem \ref{UniquenessReflecting}) and that the cases of a reflecting or an absorbing boundary emerge as limiting cases when sending the parameter of elastic killing $\kappa$ to $0$ or $+\infty$, respectively (Theorem \ref{CaseConvergenceThm}). Finally, towards the end of the paper, we show that our arguments for identifying a limiting SPDE extend to drifts depending continuously on the empirical measures (Theorem \ref{NonLinExistenceThm}), and we show that the resulting limiting solutions have McKean--Vlasov type probabilistic representations (Theorem \ref{probrep}). However, it does not seem straightforward for our $H^{-1}$ approach to yield uniqueness in that setting.

\subsection{Making sense of the elastic boundary at the origin}\label{sect:elastic_intro}

In this section, we present some heuristic arguments that motivate our notion of solution in Section 2.1 and highlight the reasons for working with a relaxed interpretation of the boundary condition \eqref{eq:classical_SPDE_BC}.

If we had a smooth noise $W^0$ and \eqref{eq:classical_SPDE}-\eqref{eq:classical_SPDE_BC} hold in the classical sense, then they imply a loss of mass at the rate
\begin{equation}\label{eq:loss_mass_classical}
\frac{d}{dt}\nu_t([0,\infty)) =  - \kappa \frac{\sigma^2(t,0)}{2} V_t(0), \quad \text{for all} \quad t> 0,
\end{equation}
almost surely, after integration by parts. This reveals a natural and slightly weaker way of imposing the elastic boundary condition without asking for spatial differentiability.

Nonetheless, we will have to drop the differentiability in time implied by \eqref{eq:loss_mass_classical}, and, moreover, we will have to avoid evaluating the solution at the boundary, for reasons we discuss below. Thus, rather than taking \eqref{eq:loss_mass_classical} as a definition we derive a relaxed version for the case where the noise $W^0$ is a Brownian motion.  Consider the density $G^{E,\kappa}_\varepsilon(x,y)$, at time $\varepsilon>0$, of a Brownian motion on $[0,\infty)$ started at $y\geq 0$ and killed elastically at $0$ with rate $\kappa>0$ (see \eqref{ElasticKernel}). Then $y\mapsto G_{\varepsilon}^{E,\kappa}(0,y)$ approximates a Dirac mass at $0$ from the right, tending to infinity at $y=0$, while the function
\begin{equation} \label{elasticTestFunc}
\phi^{E,\kappa}_\varepsilon(x) := \int_0^{\infty}  G^{E,\kappa}_{\varepsilon}(x,y)\mathbbm{1}_{[0,\infty)}(y)dy = \int_0^{\infty}  G^{E,\kappa}_{\varepsilon}(x,y)dy.
\end{equation}
defines a particular smooth approximation of $ \mathbbm{1}_{[0,\infty)}$ which is close to $1$ at $x=0$. In Theorem \ref{MeasureSPDE}, we introduce a succinct weak formulation of \eqref{eq:classical_SPDE}-\eqref{eq:classical_SPDE_BC} that will hold for the limit points of our particle system. As we will see in Section \ref{subsec:weak_boundary}, a key step in our uniqueness proof is that these weak solutions can be shown to satisfy 
\begin{equation}\label{eq:weak_BC}
d\langle \nu_t, \phi^{E,\kappa}_{\varepsilon} \rangle  = - \kappa \langle \nu_t,\frac{\sigma_t^2}{2}G_{\varepsilon}^{E,\kappa}(0,\cdot)\rangle dt +  \langle \nu_t,\mu_t \bar{g}_{\varepsilon}^{E,\kappa}\rangle dt + \langle \nu_t,\rho_t \sigma_t \bar{g}_{\varepsilon}^{E,\kappa}\rangle dW_t^0
\end{equation}
for any $\varepsilon>0$, for a suitable correction function  $\bar{g}_{\varepsilon}^{E,\kappa}:=g_\varepsilon^{E,\kappa}(0,\cdot)$ defined in \eqref{ElasticKernel}.
\begin{remark}
Through out we use the notation $\langle \xi, \phi \rangle$ for the application of a distribution $\xi$ to the test function $\phi$. In the case where $\xi$ is a measure this represents the integration of the test function against the measure.
\end{remark}
Equation \eqref{eq:weak_BC} may be viewed as a relaxed version of \eqref{eq:loss_mass_classical}, where we have shifted attention from the boundary to arbitrarily small neighbourhoods of the boundary, at the cost of introducing two correction terms: an absolutely continuous term due to the drift $\mu$ and a local martingale term due to the driving noise $W^0$. As per the arguments in Section \ref{subsec:weak_boundary}, we can take an expectation and send $\varepsilon$ to zero, to obtain
\begin{equation}\label{eq:loss_mass_expectation}
\mathbb{E}[\nu_t([0,\infty))]  = 1 - \lim_{\varepsilon\downarrow 0}  \int_0^t \kappa  \mathbb{E}\bigl[\langle \nu_s,\frac{\sigma_s^2}{2}G_{\varepsilon}^{E,\kappa}(0,\cdot)\rangle\bigr] ds.
\end{equation}
Therefore, we may instructively think of \eqref{eq:weak_BC} as enforcing \eqref{eq:loss_mass_classical} in a generalized sense in time and space, subject to taking expectations, when we assume nothing about the existence or regularity of a density or the noise being smooth.

While this is a much weaker formulation of the boundary condition \eqref{eq:classical_SPDE_BC} than \eqref{eq:loss_mass_classical}, it will be enough for us to carry through our uniqueness proof, as we see in Section \ref{subsec:uniqueness_proof}. Besides allowing for uniqueness, the relaxed formulation \eqref{eq:weak_BC} is crucial for two reasons. Firstly, we need a sufficiently relaxed formulation so that it can be guaranteed to be satisfied by weak limit points of the particle system, for which it would seem intractable to ask for much more than the above.  Secondly, even if we had more precise knowledge of the density for the limiting solutions, it would not be possible to have \eqref{eq:loss_mass_classical} holding for all $t\in[0,T]$ with probability 1. As it turns out, for each realisation, the density can blow up as we approach the boundary for certain times (albeit a set of times of measure zero). We discuss this further in Section 2.1 together with our main results.

Intuitively, suppose $dW^0_t$ is badly behaved at some time $t$, pushing a non-negligible amount of mass towards the boundary in an infinitesimal amount of time while not pushing any mass the other way (due to a Brownian path for which $B_s-B_t<0$ for all $s$ in some small right-neighbourhood of $t$). Since only a fraction of this mass leaves the domain, the rest must be instantaneously accommodated for near the boundary, which may then force $\limsup_{\varepsilon\downarrow 0}\langle \nu_t , G^{E,\kappa}_\varepsilon(0,\cdot)\rangle=+\infty$. In particular, this will also imply non-differentiability of the loss of mass at that time. As the paths $t\mapsto \nu_t([0,\infty))$ are decreasing, they are automatically differentiable at almost every time, but, for any given realisation, there can still be a dense set of times of measure zero where this fails for the aforementioned reasons.

\subsection{Literature on SPDEs with noisy boundary conditions}\label{intro:existing_litt}

Our focus in this paper is on identifying the limit of the empirical measures \eqref{eq:empirical_measures} as the unique solution of a suitable weak formulation of the SPDE \eqref{eq:classical_SPDE}-\eqref{eq:classical_SPDE_BC}. Working with the relaxed condition \eqref{eq:weak_BC}, we will make no attempt at understanding \eqref{eq:classical_SPDE_BC} in a stronger sense. Nevertheless, it is worth emphasising that there are several results in the literature on SPDEs with irregular noise terms in the boundary condition.

Al{\`o}s and Bonaccorsi \cite{Alos2002} study SPDEs with a white noise Dirichlet boundary condition on the half-line, using the theory of fundamental solutions for linear stochastic evolution equations from \cite{Nualart2000} and techniques from Malliavin calculus. They prove existence of solutions in a weighted $L^p$ space and determine the blow-up rate at the boundary. Da Prato and Zabczyk \cite{DaPratoZabczyk1993} study nonlinear stochastic evolution equations with white noise boundary conditions. They employ a stochastic version of the semigroup approach developed in \cite{Balakrishnan} to show global existence and uniqueness of mild solutions. Moreover, DaPrato and Zabczyk establish that the solution in the Neumann case exhibits a higher regularity near the boundary. Maslowski \cite{Maslowski1995} uses similar techniques based on the semigroup approach to analyse SPDEs on a bounded domain driven by space-dependent Gaussian noise and Robin-type boundary condition with bounded operators and independent noise at the boundary and shows existence and uniqueness of mild solutions. Sowers \cite{Sowers1994} studies stochastic reaction diffusion equations on an $n$-dimensional manifold with additive noise in the boundary conditions showing existence and uniqueness of solutions. In work on stochastic dynamical boundary conditions, Chueshow and Schmalfus \cite{ChueshovSchmalfuss04} \cite{ChueshovSchmalfuss07} show well-posedness of a system of quasi-linear parabolic SPDEs on bounded domains with noisy boundary dynamics, and verify that solutions give rise to a random dynamical system. In more recent work, Shirikyan \cite{Shirikyan21} studies 2d Navier--Stokes systems on a bounded domain driven by boundary noise with a piecewise independent structure.

The SPDE  \eqref{eq:classical_SPDE}-\eqref{eq:classical_SPDE_BC}  differs from these problems in a number of ways. The equation involves a stochastic transport term where the derivative of the solution appears in front of the Brownian motion. Moreover, the boundary noise is the same as the noise driving the equation in the interior while most of the literature is focused on the case of an independent noise at the boundary. 

\section{Convergence to an SPDE with elastic boundary}

To have a unique strong solution of the finite particle system and establish the desired results on convergence and well-posedness of the limit, we impose the following assumptions on the initial condition and the coefficients of the system.

\begin{assumption}[Initial Condition]\label{AssumptionInitial}
 The initial condition $\nu_0$ is supported on $(0, \infty)$, and satisfies
    \begin{equation*}
        \nu_0(\lambda, \infty) = o(\exp(-\alpha \lambda)) \quad \text{ as } \lambda \to + \infty,
    \end{equation*}
    \noindent
    for every $\alpha >0$. Furthermore, we assume that $\nu_0$ has an $L^2$-density $V_0$.
\end{assumption}

\begin{assumption}[Regularity in Space and Time] \label{AssumptionSpace} 
The coefficients $\mu$, $\sigma$ and $\rho$ have the following regularity in the space and time variables 
\begin{enumerate}
     \item  The maps  $x \mapsto \mu(t,x)$ and  $x \mapsto\sigma(t,x)$ are in $C^2([0,\infty))$ with
    \begin{equation*}
        \lvert \partial_x^n \mu(t,x) \rvert, \lvert \partial_x^n \sigma(t,x) \rvert \leq C_{\sigma,\mu}
    \end{equation*}
    \noindent
    for all $t \in [0,T], x \in \R$ and $n= 0,1,2$, for a constant $C_{\sigma,\mu}>0$,
    \item  For all $t \in [0,T], x \in \R$, $\sigma(t,x) \geq C^{-1}_{\sigma,\mu} > 0$ and $ 0 \leq \rho(t) < 1$,
    \item  The map $t \mapsto \sigma(t,x)$ is in $C^1([0,T])$ and  $\sup_{s \in [0,T]} \int_0^{\infty} \lvert \partial_t \sigma(s,y) \rvert dy < \infty$.
\end{enumerate}
\end{assumption}

Given these assumptions, strong existence and uniqueness for the system \eqref{particle SDE} follows by classical results. 
Next, we define a class of processes capturing the fundamental regularity conditions satisfied by limit points of the empirical measure processes. Our uniqueness result will be for solutions to the limiting SPDE belonging to this class.

\begin{defn}[The class $\Lambda$] \label{regularityClass}
We say that a distribution-valued c\`adl\`ag process $(\nu_t)_{t \in [0,T]}$ is of class $\Lambda$ if $\nu$ takes values in the space of sub-probability measures $\mathbf{M}_{\leq 1} (\R)$ and the following conditions are satisfied:
\begin{enumerate}
    \item \label{assptSupp} (Support on positive half-line) For every $t \in [0,T]$, the sub-probability measure $\nu_t$ is supported on the positive half-line $[0,\infty)$, \\[-15pt]
    \item \label{assptTail} (Exponential tails) For every $\alpha >0$, we have 
    \begin{equation*}
        \mathbb{E} \left[ \int_0^T \nu_t (\lambda,+ \infty) dt \right] = o(e^{- \alpha \lambda}), \quad \text{as } \lambda \to \infty, 
    \end{equation*}
    \item \label{assptBound} (Boundary decay) There exists $\gamma>0$ such that
    \begin{equation*}
        \mathbb{E} \left[ \int_0^T (\nu_t(0,\varepsilon))^2 dt \right] = O(\varepsilon^{1+\gamma}), \quad \text{as } \varepsilon \to 0,
    \end{equation*}
    \item \label{assptSpat} (Spatial concentration) There exist constants $C>0$ and $\delta \in (0,1)$  such that
    \begin{equation*}
        \mathbb{E}  \left[ \int_0^T \nu_t(a,b) dt \right] \leq C \lvert b-a \rvert^{\delta}, \quad \text{for all pairs } 0< a < b.
    \end{equation*}
\end{enumerate}

\end{defn}

\subsection{Functional convergence and well-posedness of the SPDE}

As discussed in the introduction, we are interested in a functional limit theorem for the empirical measures $\nu^N=(\nu^N_t)_{t\in[0,T]}$ seen as c\`adl\`ag stochastic processes. To this end, we proceed as in \cite{HamblyLedger2017, HamblySojmark2019, Ledger2016} and establish weak convergence on the Skorokhod space of $\mathcal{S}^\prime$-valued c\`adl\`ag paths, denoted $D_{\mathcal{S}^\prime}=D_{\mathcal{S}^\prime}[0,T]$. As usual, $\mathcal{S}^\prime$ is the space of tempered distributions, forming the dual of $\mathcal{S}$, the space of Schwarz functions on $\mathbb{R}$. As per \cite{Ledger2016}, we endow $D_{\mathcal{S}^\prime}$ with Skorokhod's M1 topology and the corresponding Borel $\sigma$-field.

\begin{theorem}[Functional limit theorem]\label{existenceThm}
Let $\nu^N=(\nu^N_t)_{t\in[0,T]}$ be given by \eqref{eq:empirical_measures} with Assumptions \ref{AssumptionInitial}-\ref{AssumptionSpace} in place. Every subsequence of $(\nu^N,W^0)$ has a further subsequence converging in law on $(D_{\mathcal{S}^{\prime}},\text{M1}) \times (C_{\R},\norm{}_{\infty})$. Moreover, for any limit point $(\nu,W^0)$, the marginal $\nu$ is in the class $\Lambda$ and there is a filtration $\mathcal{F}^{\nu,W^0}$, for which the marginal $W^0$ is a Brownian motion and $\nu$ is adapted, so that the pair $(\nu,W^0)$ satisfies the SPDE
\begin{align} \label{MeasureSPDE}
    \begin{split}
        \langle \nu_t,\phi \rangle = \langle \nu_0, \phi \rangle &+ \int_0^t  \langle \nu_s,\mu(s,\cdot) \phi^\prime \rangle ds + \frac{1}{2} \int_0^t \langle \nu_s, \sigma^2(s,\cdot)  \phi^{\prime \prime} \rangle ds\\
        &+ \int_0^t \langle \nu_s, \rho(s) \sigma(s,\cdot) \phi^\prime \rangle dW_s^0
    \end{split}
\end{align}
for all times $t \in [0,T]$ and all test functions $\phi \in \mathcal{C}^{E,\kappa}_0(\mathbb{R})$, with probability 1, where
\[
   \mathcal{C}^{E,\kappa}_0(\mathbb{R}) = \{\phi \in \mathcal{S}: \partial_x \phi(0) = \kappa \phi(0)  \}.
\]
\end{theorem}

Our next result establishes uniqueness for measure-valued solutions to the SPDE \eqref{MeasureSPDE} in the regularity class $\Lambda$. In particular, we thus obtain the full convergence in law of the empirical measure processes to the unique solution of \eqref{MeasureSPDE} in this class.

\begin{theorem}[Uniqueness] \label{uniqueness}
Let $(\nu,W^0)$ and $(\tilde{\nu},\tilde{W}^0)$ be solutions to the SPDE \eqref{MeasureSPDE} such that $\nu$ and $\tilde{\nu}$ are in the class $\Lambda$, for Brownian motions $W^0$ and $\tilde{W}^0$. Then, we have
\begin{enumerate}
    \item pathwise uniqueness if $W^0=\tilde{W}^0$, that is,
    \begin{equation}
        \nu_t(S) = \tilde{\nu}_t(S) \quad \text{ for every } t \in [0,T] \text{ and every Borel sets } S \subseteq \R.
    \end{equation}
    \item uniqueness in law, that is, $\mathrm{Law}((\nu,W^0))= \mathrm{Law}((\tilde{\nu},\tilde{W}^0))$.
\end{enumerate}
\end{theorem}

So far, our existence and uniqueness statements have been phrased solely in terms of properties of measure-valued solutions. Our third result concerns the extent to which we can guarantee $L^2$-regularity of the unique solution to the SPDE in the class $\Lambda$.

\begin{theorem}[$L^2$-regularity]\label{thm:L2-density}
	Let $(\nu,W^0)$ be the unique solution to the SPDE \eqref{MeasureSPDE} in the class $\Lambda$. With probability 1, it holds for all $t\in[0,T]$ that the measure $\nu_t$ restricted to $(0,\infty)$ has a density $V_t$, which is in $L^2$ away from the origin. Furthermore, it holds for almost every $t\in[0,T]$ that $V_t$ is the density of $\nu_t$ on all of $[0,\infty)$, and we have
	\[
	\mathbb{E}\Bigl[\int_0^T \Vert V_t \Vert_{L^2(0,\infty)}dt \Bigr]  <\infty.
	\]
\end{theorem}

The above theorem highlights two interesting issues. Firstly, for a set of times of measure zero, there could fail to be a density of $\nu_t(\omega)$ on $[0,\infty)$. Secondly, even if we have a density for all times, it could fail to be in $L^2$ up to the boundary. The former turns out to not be an issue, as results of \cite{Burdzy2003} on Brownian motion reflected off of a path of another Brownian motion will guarantee that there cannot be a point mass at zero. Moreover, based on the results of \cite{Burdzy2002}, we should expect to have $\limsup_{x\downarrow 0}V_t(x)=+\infty$ for a dense set of times. In line with this, we believe it is unlikely that $L^2$ integrability of the density can be guaranteed all the way up to the boundary for those times.

\begin{proposition}[Density on all of $[0,\infty)$]\label{prop:Burdzy}
	With probability 1, there is no atom of $\nu_t$ at the origin, for any $t\in[0,T]$. Consequently, each $V_t$ from Theorem \ref{thm:L2-density} is in $L^1(0,\infty)$ and gives the density of $\nu_t$ on $[0,\infty)$ for all $t\in[0,T]$, with probability 1.
\end{proposition}

While all our other results are proved in a self-contained manner, using mollification and energy estimates, we were not able to exploit this technique to rule out atoms at the origin, and hence a result of 
\cite{Burdzy2003} is needed for the above proposition. We do not use this proposition for any of our other results, but we believe it is important to emphasise that there is indeed a density on all of $[0,\infty)$. In particular, we note that our proof of Theorem \ref{existenceThm} deals explicitly with the a priori possibility of having atoms at zero.

\subsection{Reflecting and absorbing boundaries as limiting cases}

The elastic boundary condition in the SPDE \eqref{MeasureSPDE} is encoded in the space of test functions. Changing the space of test functions, we in turn obtain an evolution equation with a different boundary behaviour. Specifically, the space of test functions
\begin{equation}
\mathcal{C}^{A}_0(\mathbb{R}) \coloneqq \{\phi \in \mathcal{S}(\R) : \phi(0)=0 \}.
\end{equation}
corresponds to an absorbing boundary, while the space of test functions
\begin{equation}
    \mathcal{C}^{R}_0(\mathbb{R}) \coloneqq \{ \phi \in \mathcal{S}(\R) : \partial_x \phi (0) = 0 \}
\end{equation}
captures a reflecting boundary. The well-posedness of \eqref{MeasureSPDE} with a reflecting boundary can be inferred from our proofs in this paper, and so we also get the following result.

\begin{theorem} \label{UniquenessReflecting}
The SPDE \eqref{MeasureSPDE} with a reflecting boundary at zero, captured by the test function space $\mathcal{C}^R_0(\R)$, has a unique solution in the class $\Lambda$, as in Theorem \ref{uniqueness}. 
\end{theorem}

Existence and uniqueness in the case of an absorbing boundary has been studied in \cite{HamblyetAl201, Ledger2014, HamblyLedger2017, HamblySojmark2019} with various additional features. Our final result shows that the elastic boundary gives rise to the absorbing and reflecting boundaries as limiting cases.

\begin{theorem} \label{CaseConvergenceThm}
Let $(\nu^\kappa,W^0)$ denote the unique weak solution with $\nu$ in the class $\Lambda$ to the SPDE \eqref{MeasureSPDE} with elastic boundary for a given $\kappa >0$. Then we have that
\begin{enumerate}
    \item as $\kappa \to \infty$, $(\nu^\kappa,W^0)$ converges in law to $(\nu^\infty,W^0)$ which is a solution to the SPDE (\ref{MeasureSPDE}) with test function space $\mathcal{C}^A_0(\mathbb{R})$,
    \item as $\kappa \to 0$, $(\nu^\kappa,W^0)$ converges in law to $(\nu^0,W^0)$ which is a solution to the SPDE (\ref{MeasureSPDE}) with test function space $\mathcal{C}^R_0(\mathbb{R})$.
\end{enumerate}
\end{theorem}

\section{Probabilistic estimates for the particle system}\label{probEstSection}

In this section we show a series of probabilistic estimates that will be used in Section \ref{regularitySection} to prove the regularity of limit points. As a first step, we note that all particles in the system  $\{(X^{i}_t \mathbbm{1}_{t < \tau^i})_{t \in [0,T]} \}_{1 \leq i \leq N}$ are identically distributed. As a result, we obtain for an arbitrary measurable set $S \subseteq \R$, $N \geq 1$ and $t \in [0,T]$ 
\begin{equation} \label{finite expecation}
    \mathbb{E} \left[ \nu^N_t(S) \right] = \frac{1}{N} \sum_{i=1}^N \mathbb{E}\left[\mathbbm{1}_{\{X_t^{i}\in S\}} \mathbbm{1}_{t < \tau^1}\right] = \mathbb{P}(X_t^{i} \in S, t < \tau^1). 
\end{equation}
We  show how we can use a scale transformation and a change of measure to estimate (\ref{finite expecation}) in terms of the distribution of a reflected Brownian motion $\lvert W \rvert$. 

\begin{lemma}[Scale transformation] \label{Scale}
Let $X_t^{1}$ be a particle from the particle system. Define the transformation function $\zeta: [0,T] \times \R \to \R$ as
\begin{equation*}
    \zeta(t,x) \coloneqq \int_0^x \frac{dy}{\sigma(y,t)}
\end{equation*}
then the process $Z_t \coloneqq \zeta(t,X_t^{1})$ is a semimartingale and its dynamics are given by $dZ_t = \Tilde{\mu}_tdt + dB_t  + dL_t(Z)$. The stochastic process B is a Brownian motion
\begin{equation*}
    B_t = \int_0^t \rho_s dW_s^0 + \int_0^t (1- \rho_s^2)^{\frac{1}{2}}dW_s^1,
\end{equation*}
\noindent
L is the local time of Z at zero and $\tilde{\mu}$ a drift coefficient given by
\begin{equation*}
    \Tilde{\mu}_t = (\frac{\mu}{\sigma}- \partial_x \sigma)(t,X_t^{1}) - \int_0^{X_t^{1}} \frac{\partial_t \sigma}{\sigma^2}(t,y)dy.
\end{equation*}
The coefficient $\tilde{\mu}_t$ is uniformly bounded in N and t.
\end{lemma}

\begin{proof}
We apply It{\^o}'s formula to obtain the dynamics of $Z$
\begin{align*}
    dZ_t &= \partial_t \zeta(t,X_t^{1}) dt + \partial_x \zeta(t,X_t^{1})dX_t^{1}+ \partial_{xx}\zeta(t,X_t^{1})d[X^{1}]_t \\
    &= \Tilde{\mu}(t,X_t^{1})dt+ dB_t + \frac{1}{\sigma(t,X_t^{1})}dL_t
\end{align*}
with $\Tilde{\mu}$ as defined above. By \cite[Proposition 4.2 and Proposition 4.4]{SojmarkWorkingPaper22}  we can identify the last term as 
\begin{equation*}
    \int_0^t \frac{1}{\sigma(s,X_s^{1})}dL_s  = L(Z_t)
\end{equation*}
establishing the equation for Z. The uniform boundedness for the drift parameter $\Tilde{\mu}$ follows from the conditions on the coefficients in Assumption \ref{AssumptionSpace}.
\end{proof}

\begin{remark}
The boundedness assumptions on $\sigma$ give that, for all $t \in [0,T]$,
\begin{equation} \label{zetaBound}
    \left \lvert \zeta(t,y) - \zeta(t,x) \right \rvert \leq C_{\mu,\sigma} \lvert y-x \rvert, \qquad x,y \geq 0
\end{equation}
with the constant $C_{\mu,\sigma}>0$ from Assumption \ref{AssumptionSpace}. Moreover, note that the transformation $\zeta(t,\cdot)$ is invertible for all $t \in [0,T]$.
\end{remark}

\begin{remark}
Note that the process $Z_t$ obtained from the scale transformation is a reflected Brownian motion with drift. Thus, the process Z solves the Skorokhod problem for the process $dY_t =  \Tilde{\mu}dt + dB_t$ and, in particular, stays non-negative (see \cite{revuz}).
\end{remark}

After removing the volatility factor $\sigma$ through the scale transformation, the next step is to remove the drift and obtain a reflected Brownian motion. This can be done using Girsanov's theorem. 
\begin{lemma}[Removing the drift] \label{removeDrift}
For all $\delta \in (0,1)$ there is a constant $c_{\delta}>0$ so that
\begin{equation*}
    \mathbbm{P}(X_t^{1} \in S) \leq c_{\delta} F_t(\zeta(t,S))^{\delta}, \qquad{\text{ for every measurable } S \subseteq \R},
\end{equation*}
where $F_t$ is the marginal law of a reflected Brownian motion at time t.
\end{lemma}

\begin{proof}
We know that $Z$ is the solution to the Skorokhod problem for the process $Y$ with $dZ_t = \tilde{\mu}_t dt + dB_t + dL_t(Z)$. As $\tilde{\mu}$ is bounded, we can apply Girsanov's Theorem with the change of measure
\begin{equation*}
     \left.\frac{d\mathbb{Q}}{d\mathbbm{P}}\right|_{\mathcal{F}_t}= \exp \left(-\int_0^t \tilde{\mu}_sdB_s - \frac{1}{2} \int_0^t (\tilde{\mu_s})^2ds  \right)
 \end{equation*}
 to see that there is a measure $\mathbb{Q}$ with a Brownian motion $B^{\mathbb{Q}}$ under which the dynamics of $Z$ are those of a reflected Brownian motion
\begin{equation*}
    dZ_t = dB_t^{\mathbb{Q}} + dL_t(B^{\mathbb{Q}}).
\end{equation*}
The result then follows by proceeding as in the proof of \cite[Lemma 4.2]{HamblyLedger2017}.
\end{proof}

As an easy corollary we obtain good spatial concentration of the empirical measures.
\begin{lemma} \label{FiniteSpatial}
There exist constants $c>0$ and $\delta \in (0,1)$ such that property (iv) of Definition \ref{regularityClass} is satisfied by $(\nu_t^N)_{t\in[0,T]}$ uniformly in $N \geq 1$.
\end{lemma}

\begin{proof}
By Lemma \ref{removeDrift} we have
\begin{align*}
    \mathbb{E} \left[\nu_t^N(a,b) \right] \leq c_{\delta}F_t(\zeta(t,(a,b)))^{\delta}.
\end{align*}
We note that $\zeta(t,(a,b)) \subseteq [\zeta(t,a),\zeta(t,b)]$ where $\zeta(t,(a,b)) = \{\zeta(t,x): x \in (a,b) \}$. Using the reflecting heat kernel $G_{\varepsilon}^R$ given in (\ref{reflectingKernel})  we estimate
\begin{align*}
    F_t(\zeta(t,(a,b))) & \leq \int_0^{\infty} \int_{\zeta(t,a)}^{\zeta(t,b)} \frac{1}{\sqrt{2 t \pi}} \Bigg[\exp \left(-\frac{(x-\zeta(0,x_0))^2}{2t} \right)\\
    & \qquad \qquad + \exp \left(-\frac{(x+\zeta(0,x_0))^2}{2t} \right) \Bigg]dx \nu_0(dx_0) \\
    & \leq 2 (2 \pi t)^{-1/2} (\zeta(t,b)-\zeta(t,a)) \leq 2C_{\mu,\sigma} (2 \pi t)^{-1/2} (b-a)
\end{align*}
where the last inequality follows from (\ref{zetaBound}). Since the map $ t \mapsto t^{-\delta/2}$ is integrable, the result then follows.
\end{proof}

Next we prove the boundary estimate that is a crucial part of the proof of uniqueness for the SPDE. First, we establish an estimate in the simple case where the particles follow a reflected Brownian motion with zero drift and $\sigma \equiv 1$.
\begin{lemma} \label{twoRefBM}
Assume that we have two particles X and Y following a reflected Brownian motion with correlation $\rho $ with $\lvert \rho \rvert <1$ and initial values $X_0$ and $Y_0$. i.e. $X = \lvert X_0 + W^1 \rvert$ and $Y= \lvert Y_0 + W^2 \rvert$ Then we have the following estimate
\begin{equation}
    \mathbb{P}(0 < \lvert X_0 + W_t^1 \rvert < \varepsilon,0< \lvert Y_0 + W_t^2 \rvert < \varepsilon ) \leq \frac{2}{ \pi \sqrt{1-\rho^2}t} \varepsilon^2.
\end{equation}
\end{lemma}

\begin{proof}
We have, using the explicit form of the folded bivariate Gaussian density,
\begin{align*}
    &\mathbb{P}(0 < \lvert X_0 + W_t^1 \rvert < \varepsilon,0< \lvert Y_0 + W_t^2 \rvert < \varepsilon ) \\ &= \int_0^{\varepsilon} \int_0^{\varepsilon}  \int_{0}^{\infty} \int_{0}^{\infty }\frac{1}{2 \pi \sqrt{1-\rho^2}t} \times \\
    & \Bigg[ \exp\left(- \frac{1}{2 (1-\rho^2)} \left(\frac{(x-x_0)^2}{t}- 2 \rho \frac{(x-x_0)(y-y_0)}{t}+ \frac{(y-y_0)^2}{t} \right) \right) \\
    &+ \exp\left(- \frac{1}{2 (1-\rho^2)} \left(\frac{(x+x_0)^2}{t}- 2 \rho \frac{(x+x_0)(y+y_0)}{t}+ \frac{(y+y_0)^2}{t} \right) \right) \\
    &+ \exp\left(- \frac{1}{2 (1-\rho^2)} \left(\frac{(x+x_0)^2}{t}+ 2 \rho \frac{(x+x_0)(y-y_0)}{t}+ \frac{(y-y_0)^2}{t} \right) \right)\\
    &+ \exp\left(- \frac{1}{2 (1-\rho^2)} \left(\frac{(x-x_0)^2}{t} + 2 \rho \frac{(x-x_0)(y+y_0)}{t}+ \frac{(y+y_0)^2}{t} \right) \right)
    \Bigg]
    \nu_0(dx_0) \nu_0(dy_0) dxdy
\end{align*}
Simple bounds on the Gaussian density give the result
\end{proof}

\begin{proposition}[Boundary estimate] \label{boundaryEst}
For any  $q > 1$ and $t \in [0,T]$ we have as $\varepsilon \to 0 $ that for $i,j = 1, \ldots, N$ with $i \neq j$ 
\begin{equation*}
    \mathbbm{P}(0 < X_t^i < \varepsilon,t < \tau^i, 0 < X_t^j < \varepsilon, t < \tau^j) = t^{-\frac{1}{q}} O(\varepsilon^{\frac{2}{q}}).
\end{equation*}
\end{proposition}

\begin{proof}
It is clear that we can estimate the probability from above by dropping the stopping time conditions
\begin{equation*}
    \mathbbm{P}(0 < X_t^i < \varepsilon,t < \tau^i, 0 < X_t^j < \varepsilon, t < \tau^j) \leq \mathbbm{P}(0 < X_t^i < \varepsilon, 0 < X_t^j < \varepsilon).
\end{equation*}
We apply the scale transformation $\zeta$ and then use the multidimensional version of Girsanov's theorem (see \cite[\S1.7.4]{jeanblanc}) that preserves the correlation structure between the Brownian motions. Note  furthermore that $\zeta(t,0) = 0$ and by (\ref{zetaBound})  we have $\zeta(t,\varepsilon) \leq C_{\mu,\sigma} \varepsilon$. As a result, we have 
\begin{equation*}
    \mathbbm{P}(0 < X_t^i < \varepsilon, 0 < X_t^j < \varepsilon) \leq \tilde{C} \mathbb{Q}(0 < \lvert X_0^i+ W_t^i \rvert < C_{\mu,\sigma} \varepsilon, 0 < \lvert X_0^j+ W_t^j \rvert < C_{\mu,\sigma} \varepsilon )^{1/q}
\end{equation*}
for $q >1$ and a constant $\tilde{C}>0$. Then, by Lemma \ref{twoRefBM} we get the desired estimate. 
\end{proof}
Note that $t \mapsto t^{-\frac{1}{q}}$ is integrable because $q >1$. A decay result can also be found for the mass escaping to infinity.
\begin{proposition}[Tail estimate] \label{tail}
For every $\alpha >0$, as $\lambda \to + \infty$
\begin{equation*}
    \mathbb{E}  \left[\nu^N_t(\lambda,+\infty) \right] = o(\exp(-\alpha \lambda)), \quad \text{uniformly in $N \geq 1$ and $t \in [0,T]$}.
\end{equation*}
\end{proposition}

\begin{proof}
We will again rely on the results and notation of Lemma \ref{removeDrift}. Setting $\pi_0 = \nu_0 \circ \zeta(0,\cdot)^{-1}$ we can estimate for $\lambda >0$
\begin{align*}
    F_t((\lambda,\infty)) &= \int_0^{\infty} \mathbb{P}(\lvert W_t \rvert > \lambda \vert \lvert W_0 \rvert = x_0) \pi_0(dx_0) \\
    &= \frac{1}{\sqrt{2 \pi t}} \int_0^{\infty} \int_{\lambda}^{\infty}  \left[ e^{- \frac{(x-x_0)^2}{2t}}+ e^{- \frac{(x+x_0)^2}{2t}} \right]dx \pi_0(dx_0) \\
    &\leq \frac{2}{\sqrt{2 \pi t}} \int_0^{\infty} \int_{\lambda}^{\infty} e^{- \frac{(x-x_0)^2}{2t}} dx d\pi_0(dx_0) = 2 \int_0^{\infty} \mathbb{P}( W_t  > \lambda \vert  W_0  = x_0) \pi_0(dx_0).
\end{align*}

Now we can argue as in \cite[Proposition 4.5]{HamblyLedger2017}  using Assumption \ref{AssumptionInitial} for the initial condition and applying Lemma \ref{removeDrift}. 
\end{proof}

\section{Convergence of the particle system} \label{existenceSection}

In this section we show that limit points of the particle system satisfy the SPDE \eqref{MeasureSPDE}. Furthermore, we verify that all limit points satisfy the regularity conditions of the class $\Lambda$ and thus complete the proof of Theorem \ref{existenceThm}.

\subsection{Evolution equation for the empirical measures}

We begin by finding an evolution equation for the empirical-measure process $\nu^N$ of the particle system.
\begin{proposition}[Evolution equation for $\nu^N$] \label{finiteEvolEq}
For all $N \geq 1$ the empirical measure process $\nu^N$ satisfies the evolution equation
\begin{align} \label{eq:eefornu}
    \begin{split}
        \langle \nu_t^N, \phi \rangle =& \langle \nu_0^N, \phi \rangle + \int_0^t \langle \nu_s^N, \mu(s,\cdot) \phi^{\prime} \rangle ds + \frac{1}{2} \int_0^t \langle \nu_s^N, \sigma^2(s,\cdot) \phi^{\prime \prime} \rangle ds \\
        &+ \int_0^t \langle \nu_s^N, \rho_s \sigma(s,\cdot) \phi^\prime \rangle dW_s^0 + I_t^N(\phi) + J_t^N(\phi),
    \end{split}
\end{align}
with
\begin{equation*}
    I_t^N(\phi) = \frac{1}{N} \sum_{i=1}^N  \int_0^{t \wedge \tau^i} (1-\rho_s^2)^{\frac{1}{2}} \sigma(s,X_s^i) \phi^\prime(X_s^i)  dW_s^i
\end{equation*}
and
\begin{equation*}
    J_t^N(\phi) = \phi(0) \left(\frac{\kappa}{N} \sum_{i=1}^N    L_{t \wedge \tau^i}^i - \frac{1}{N} \sum_{i=1}^N  \mathbbm{1}_{\tau^i \leq t}\right),
\end{equation*}
for all $\phi \in \mathcal{C}^{E,\kappa}_0(\mathbb{R})$, where $I_t^N(\phi)\to 0$ and $J_t^N(\phi) \to 0$ as $N \to \infty$.
\end{proposition}

\begin{proof}
Apply It{\^o}'s formula to $\phi(X_{t \wedge \tau^i}^i)$ to get
\begin{align*}
    &\phi(X_{t \wedge \tau^i}^i) = \phi(X_0^i) + \int_0^{t \wedge \tau^i} 
    \!\!\phi^\prime(X_s^i) \mu(s,X_s^i)ds + \int_0^{t \wedge \tau^i} \!\! \phi^\prime(X_s^i) (1-\rho_s^2)^{\frac{1}{2}} \sigma(s,X_s^i)dW_s^i \\
    &\quad+ \int_0^{t \wedge \tau^i}\!\! \phi^\prime(X_s^i) \rho_s \sigma(s,X_s^i)dW_s^0
    + \frac{1}{2} \int_0^{t \wedge \tau^i}\!\! \phi^{\prime \prime}(X_s^i) \sigma^2(s,X_s^i) ds + \int_0^{t \wedge \tau^i} \!\!\phi^\prime(X_s^i)dL_s^i.
\end{align*}
Recall that 
\begin{equation*}
    \langle \nu_t^N, \phi \rangle = \frac{1}{N} \sum_{i=1}^N \phi(X_t^i) \mathbbm{1}_{t < \tau^i}.
\end{equation*}
Using $ \phi(X_{t \wedge \tau^i}) = \phi(X_{t}) \mathbbm{1}_{t < \tau^i} + \phi(X_{\tau^i})\mathbbm{1}_{t \geq  \tau^i}$, the elastic boundary condition of the test function $\phi$ and integrating against $\nu^N$ we obtain \eqref{eq:eefornu}.

It remains to prove that the two terms $I^N(\phi)$ and $J^N(\phi)$ vanish in a suitable way as $N \to \infty$. We start by computing the quadratic variation of $I^N(\phi)$. By independence of the idiosyncratic Brownian motions $W^i$ we obtain 
\begin{equation*}
    [I_{\cdot}^N(\phi)]_t = \frac{1}{N^2} \sum_{i=1}^N \int_0^t (1-\rho_s^2)  \sigma^2(s,X_s^i) (\phi^\prime(X_s^i))^2ds.
\end{equation*}
The boundedness of $\sigma^2$ and $\phi^\prime$ yields
\begin{equation*}
    [I_{\cdot}^N(\phi)]_t = \norm{\phi^\prime}_{\infty}^2 O(N^{-1}), \quad \text{as } N \to \infty.
\end{equation*}
By Doob's martingale inequality it follows that
\begin{equation*}
    \mathbb{E} \left[\sup_{t \in [0,T]} \lvert I_t^N(\phi) \rvert^2 \right] \leq 4 \mathbb{E} \left[I_T^N(\phi)^2 \right] = 4  \mathbb{E} \left[[I_{\cdot}^N(\phi)]_T \right]\to 0 \mbox{ as $N\to\infty$}.
\end{equation*}
From \cite[Theorem 2.5]{SojmarkWorkingPaper22} we obtain
\begin{equation*}
    \mathbb{E} \left[\sup_{t \in [0,T]} \lvert J_t^N(\phi) \rvert^2 \right] \to 0, \qquad {\text{as $N \to \infty$}}.
\end{equation*}
\end{proof}

\subsection{Convergence to the Limit SPDE}\label{convergenceSection}

The next step is now to show that limit points of the empirical measure processes $\nu^N$ indeed solve our SPDE. First, we need to verify tightness and conclude that there exist limit points. To do this we follow the approach in \cite{HamblyLedger2017}, showing tightness in $D_{\mathcal{S}^\prime}$, the space of c{\`a}dl{\`a}g paths with values in the space of tempered distributions, equipped with the M1 topology and then verifying that the limit points are sub-probability measure-valued processes.

\begin{proposition}\label{EmpLossProp}
For every $t \in [0,T]$ and $\eta>0$, the empirical loss process, $\mathcal{L}^N$, obeys
\begin{equation} \label{empiricalLossDefn}
 \lim_{\delta \to 0} \lim_{N \to \infty} \mathbb{P}\left(\mathcal{L}^N_{t+\delta}-\mathcal{L}_t^N \geq \eta \right)=0,\quad\mathcal{L}_t^N \coloneqq \frac{1}{N} \sum_{i=1}^N \mathbbm{1}_{\tau^i \leq t}.
\end{equation}
\end{proposition}
\begin{proof}
Using the triangle inequality we can estimate
\begin{align*}
    \mathbb{P} \left(\mathcal{L}^N_{t+\delta} - \mathcal{L}_t^N \geq \eta \right) \leq &\mathbb{P} \left( \left \lvert \mathcal{L}^N_{t+\delta} - \mathcal{L}_t^N -  \frac{1}{N\kappa} \sum_{i=1}^N\left( L_{t+ \delta \wedge \tau^i}^i -L_{t \wedge \tau^i}^i \right) \right \rvert \geq \frac{\eta}{2} \right) \\
    &+ \mathbb{P} \left( \frac{1}{N\kappa} \sum_{i=1}^N\left( L_{t+ \delta \wedge \tau^i}^i -L_{t\wedge \tau^i}^i \right) \geq \frac{\eta}{2}  \right).
\end{align*}
Applying the convergence result for $J_t^N(\phi)$ as $N \to \infty$ in Proposition \ref{finiteEvolEq}, which holds for all $\phi \in \mathcal{C}_0^{E,\kappa}(\R)$, at the two time points $t$ and $t+\delta$ we see that the first probability on the right-hand side vanishes as $N \to \infty$. Markov's inequality can be applied to the second term to obtain
\begin{equation*}
    \mathbb{P} \left( \frac{1}{N\kappa} \sum_{i=1}^N\left( L_{t+ \delta \wedge \tau^i}^i -L_{t\wedge \tau^i}^i \right) \geq \frac{\eta}{2}  \right) \leq \frac{2}{\eta \kappa} \mathbb{E} \left[L_{t+ \delta}^i - L_t^i  \right].
\end{equation*}
Since the particles $X^i$ are continuous, the local time processes $L^i$ are also continuous. Thus, we get convergence to $0$ as $\delta \to 0$ uniform in $N \geq 1$.
\end{proof}

\begin{proposition}[Tightness] \label{TightnessN}
The sequence of empirical-measure processes $(\nu^N)$ is tight on the space $(D_{\mathcal{S}^{\prime}},M1)$ and the sequence $(\nu^N,W^0)$ is tight on $(D_{\mathcal{S}^{\prime}},M1) \times (C_{\R}, \norm{\cdot}_{\infty})$.
\end{proposition}

\begin{proof}
Using \cite[Theorem 3.2]{Ledger2016} it suffices to show that, for arbitrary $\phi \in \mathcal{S}$, the process $\langle \nu^N,\phi \rangle$ is tight on $D_{\R}$ in the M1 topology. For all $t \in [0,T]$, we decompose $\langle \nu_t^N, \phi \rangle$ as
\begin{equation} \label{TightnessDecomp}
    \langle \nu_t^N, \phi \rangle = \langle \hat{\nu}^N_t,\phi \rangle - \phi(0) \mathcal{L}_t^N,\quad \hat{\nu}^N_t \coloneqq \frac{1}{N} \sum_{i=1}^N \delta_{X_{t \wedge \tau^i}^i},
\end{equation}
with $\mathcal{L}^N_t$ as defined in (\ref{empiricalLossDefn}). The term $\phi(0) \mathcal{L}_t^N$ is monotone and does not contribute to the M1-modulus of continuity. By \cite[Propositions 4.1 and 4.2]{Ledger2016} and the arguments in the proof of \cite[Proposition 5.1]{HamblyLedger2017} it is then sufficient to establish that we have 
\begin{equation} \label{FirstTightness}
    \mathbb{E} \left[ \lvert \langle \hat{\nu}_t^N, \phi \rangle -\langle \hat{\nu}_s^N, \phi \rangle \rvert^4 \right] = O(\lvert t-s \rvert^2) \quad \text{as } \lvert t-s \rvert \to 0,
\end{equation}
and that for any $\varepsilon>0$
\begin{equation} \label{SecondTightness}
    \lim_{\delta \to 0} \lim_{N \to \infty} \mathbb{P} \left(\sup_{t \leq \delta} \lvert \langle \nu_t^N-\nu_0^N, \phi \rangle \rvert + \sup_{t \in (T-\delta,T)} \lvert \langle \nu_T^N - \nu_t^N,\phi \rangle \rvert > \varepsilon \right)=0.
\end{equation}

For any $s, t \in [0,T]$ it holds that
\begin{align*}
    \mathbb{E} \left[ \lvert \langle \hat{\nu}_t^N, \phi \rangle -\langle \hat{\nu}_s^N, \phi \rangle \rvert^4 \right] &\leq \frac{1}{N} \sum_{i=1}^N \mathbb{E} \left[ \lvert \phi(X_t^i) - \phi(X_s^i)
    \rvert^4 \right]  \leq  \norm{\phi}_{Lip}^4 \mathbb{E} \left[ \lvert X_t^i - X_s^i \rvert^4 \right]
\end{align*}
where $\norm{\phi}_{Lip}$ denotes the Lipschitz constant of the function $\phi$. It remains to deal with the expectation on the right-hand side. By using the inequality $(a+b+c)^4 \leq 27 (a^4 +b ^4 + c^4)$ we obtain
\begin{align*}
    \mathbb{E} \left[ \lvert X_t^i - X_s^i \rvert^4  \right]& \leq  27 \Big( \mathbb{E} \left[\left \lvert \int_s^t \mu(u,X_u^i)du \right\rvert^4 \right] +  \mathbb{E}  \left[\left\lvert \int_s^t \sigma(u,X_u^i) dB_u^i \right\rvert^4 \right] + \mathbb{E}  \left[\left\lvert L_t^i - L_s^i \right\rvert^4 \right] \Big). 
\end{align*}
Since the coefficient $\sigma$ is bounded, the stochastic integral is a martingale and we can apply the Burkholder--Davis--Gundy inequality \cite[Theorem IV.42.1]{rogersWilliams} to find that
\begin{equation*}
    \mathbb{E} \left[ \left\lvert \int_s^t \sigma(u,X_u^i) dB_u^i \right\rvert^4 \right] = O(\lvert t-s \rvert^2), \quad \text{as } \lvert t-s \rvert \to 0.
\end{equation*}
The boundedness assumption in Assumption \ref{AssumptionSpace} $(i)$ for the drift term $\mu$  yields
\begin{equation*}
    \mathbb{E} \left[ \left \lvert \int_s^t \mu(u,X_u^i)du \right\rvert^4 \right] = O(\lvert t-s \rvert^4), \quad \text{as $ \lvert t-s \rvert \to 0$}.
\end{equation*}
The estimate $\mathbb{E}  \left[\left\lvert L_t^i - L_s^i \right\rvert^4 \right] = O(\lvert t-s \rvert^2)$, as $\lvert t-s \rvert \to 0$, follows using the local time representation from the Skorokhod problem. 
Hence, we have $\mathbb{E} [ \lvert X_t^i - X_s^i \rvert^4 ] = O(\lvert t-s \rvert^2)$ as $ \lvert t-s \rvert \to 0$, and so we obtain (\ref{FirstTightness}). 

To establish \eqref{SecondTightness} it is enough to consider the small time interval $(0,\delta)$ as the result for the interval $(T-\delta, T)$ follows similarly. We can apply (\ref{TightnessDecomp}) to (\ref{SecondTightness}) to get
\begin{equation*}
    \mathbb{P} \left(\sup_{t \leq \delta} \lvert \langle \nu_t^N-\nu_0^N, \phi \rangle \rvert > \varepsilon  \right) \leq \mathbb{P} \left(\sup_{t \leq \delta} \lvert \langle \hat{\nu}_t^N-\hat{\nu}_0^N, \phi \rangle \rvert  \geq \frac{\varepsilon}{2} \right) + \mathbb{P} \left(\lvert \phi(0) \rvert \mathcal{L}^N_{\delta} \geq \frac{\varepsilon}{2} \right).
\end{equation*}

We can then again use the estimates in the poof of \cite[Proposition 5.1]{HamblyLedger2017} combined with the arguments for \eqref{FirstTightness} and obtain that the first term converges to $0$ as $\delta \to 0$, uniformly in $N \geq 1$. Proposition \ref{EmpLossProp} provides the same result for the second term and the tightness follows.
\end{proof}

At this stage we have only established that there exists a subsequence such that $(\nu^{N_k},W^0) \to (\nu^*,W^0) $ weakly in the space of $\mathcal{S}^\prime$-valued c{\`a}dl{\`a}g processes equipped with the M1 topology. However, we can deduce that the limiting processes actually take values in the space of sub-probability measures supported on $[0,\infty)$ by the same arguments as \cite[Proposition 5.4]{HamblyLedger2017}.
\begin{proposition} \label{subProbValued}
Let $(\nu^*,W^0)$ realise the limiting law of $(\nu^{N_k},W^0)$. Then each $\nu_t^*$ is a sub-probability measure supported on $[0,\infty)$ for all $t \in [0,T]$, with probability 1.
\end{proposition}

Next, we confirm that there is a suitable filtration for the limit points.
\begin{lemma} \label{BrownianFiltration}
Let $(\nu^*,W^0)$ be the limit of a subsequence $(\nu^{N_k},W^0)$. Then there is a filtration $\mathcal{F}^{\nu^*,W^0}$ so that $\nu^*$ is adapted and $W^0$ is an $\mathcal{F}^{\nu^*,W^0}$-Brownian motion.
\end{lemma}

\begin{proof}
On the sample space $\Omega = D_{\mathcal{S}^{\prime}} \times C_{\R}$ we define the filtration
\begin{equation} \label{filtrationDefn}
    \mathcal{F}_t^{\nu^*,W^0} \coloneqq \sigma \left( \{(\nu^*,W^0) \mapsto (\langle \nu^*_s,\phi \rangle,W^0_s)  : s <t, \phi \in \mathcal{S}\} \right).
\end{equation}
This is a sensible filtration because the projections $\pi^{\phi^1,\cdot, \phi^n}_{t_1,\cdot,t_n}$ given by
\begin{equation*}
    \pi^{\phi^1,\cdot, \phi^n}_{t_1,\cdot,t_n}: D_{\mathcal{S}^\prime} \to \R^n, \pi^{\phi^1,\cdot, \phi^n}_{t_1,\cdot,t_n}(\xi) = (\xi_{t_1}(\phi_1),\ldots, \xi_{t_n}(\phi_n)),
\end{equation*}
generate the Borel sets on $D_{\mathcal{S}^{\prime}}$ as shown in \cite{Ledger2016}. It is clear that $\nu^*$ is adapted to $\mathcal{F}^{\nu^*,W^0}$. It remains to show that $W^0$ is a Brownian motion in this filtration. This follows with the same arguments as in \cite[Section 4.1]{LedgerSojmark2021}.
\end{proof}

To show that limit points of the empirical-measure process $\nu^N$ solve the SPDE we will prove that solutions of the finite evolution equation of Proposition \ref{finiteEvolEq} converge to solutions of our SPDE. This will be done using the martingale approach of \cite{HamblyLedger2017} and \cite{HamblySojmark2019}. This approach easily allows for a higher level a generality that also covers an extension considered in Section 7. The martingale components  are given by
\begin{align}
\begin{split} \label{martingaleComponents}
     M^{\phi}(\xi)(t) \coloneqq &\langle \xi_t,\phi \rangle -\langle \xi_0,\phi \rangle - \int_0^t \langle \xi_s,\mu(s,\cdot) \phi^\prime \rangle ds\\
     &- \frac{1}{2} \int_0^t \langle \xi_s,\sigma^2(s,\cdot) \phi^{\prime \prime} \rangle ds,\\
     S^{\phi}(\xi)(t) \coloneqq &M^{\phi}(\xi)(t)^2 -\int_0^t \langle \xi_s,\sigma(s,\cdot)\rho(s) \phi^\prime \rangle^2ds, \\
    C^{\phi}(\xi,w)(t) \coloneqq &M^{\phi}(\xi)(t) \cdot w(t) - \int_0^t \langle \xi_s,\sigma(s,\cdot)\rho(s) \phi^\prime \rangle ds
\end{split}
\end{align}
for $\xi \in D_{\mathcal{S}^{\prime}}$, $w \in C_{\R}$ and $\phi \in \mathcal{C}^{E,\kappa}_0(\mathbb{R})$. Using the results from \cite[Section 5]{HamblyLedger2017} we obtain the weak convergences
\begin{align*}
    M^{\phi}(\nu^{N_k})(t) &\to M^{\phi}(\nu^*)(t), \\
    S^{\phi}(\nu^{N_k})(t) & \to S^{\phi}(\nu^*)(t), \\
    C^{\phi}(\nu^{N_k},W^0)(t) &\to C^{\phi}(\nu^*,W^0)(t) \qquad{ \text{in } \R}.
\end{align*}
on a deterministic cocountable subset of $[0,T]$ for all $\phi \in \mathcal{C}^{E,\kappa}_0(\mathbb{R})$. Then, the convergence of the solutions follows as in \cite[Proposition 5.11]{HamblyLedger2017}.

\begin{proposition}[Evolution equation] \label{convergenceToEv}
 Let $(N_k)$ be a subsequence such that the weak convergence $(\nu^{N_k},W^0) \to (\nu^*,W^0)$ holds. Then, for every $\phi \in \mathcal{C}^{E,\kappa}_0(\mathbb{R})$ the process $M^{\phi}(\nu^*),S^{\phi}(\nu^*)$ and $C^{\phi}(\nu^*,W)$ are martingales. As a consequence, $\nu^*$ and $W^0$ satisfy the limit SPDE from Theorem \ref{existenceThm}.
\end{proposition}

\subsection{Regularity of the Limit Solution} \label{regularitySection}

We now show that the solutions we constructed as limits from the particle system satisfy the regularity conditions of Definition \ref{regularityClass} using the estimates from Section~\ref{probEstSection}. 

\begin{proposition} \label{limitRegularity}
Let $(\nu^*,W^0)$ be a limit point of the sequence of empirical-measure processes. Then the process $(\nu_t^*)_{t \in [0,T]}$ is in the class $\Lambda$.
\end{proposition}

\begin{proof}
In Proposition \ref{subProbValued} we already showed that $\nu^*$ takes values in the space of sub-probability measures supported on $[0, \infty)$. It remains to prove the regularity conditions on the sub-probability measures. 

Consider a finite open interval $I = (x,y) \subseteq \R$. Take $\delta >0$ and any $\phi_{\delta} \in \mathcal{S}$ such that $\phi_{\delta}=1$ on I, $\phi_{\delta}=0$ on $(-\infty,x-\delta) \cup (y+\delta,\infty)$ and $\phi_{\delta} \in (0,1)$ otherwise. From $(\nu^{N_k},W) \to (\nu^*,W)$ weakly we get that $\int_0^t \langle \nu_s^{N_k},\phi_{\delta}\rangle ds \to \int_0^t \langle \nu_s^*,\phi_{\delta}\rangle $ds weakly by \cite[Theorem 11.5.1]{whitt2002}. We have
\begin{equation*}
    \mathbb{E} \left[ \int_0^T \nu_t^*(I)dt \right] \leq \mathbb{E} \left[ \int_0^T \langle \nu_t^*,\phi_{\delta}\rangle dt \right] = \lim_{k \to \infty} \mathbb{E} \left[ \int_0^T \langle \nu_t^{N_k},\phi_{\delta} \rangle dt \right].
\end{equation*}
Considering the bound in Lemma \ref{FiniteSpatial} and taking the limit $\delta \to 0$ yields \ref{assptSpat} in the definition of the class $\Lambda$.

Using the same approach we obtain that for $\varepsilon >0$
\begin{align*}
    \mathbb{E} \left[\nu_t^*(0,\varepsilon)^2 \right] & \leq \liminf_{N \to \infty} \frac{1}{N^2}  \sum_{i=1}^N \sum_{j \neq i,j=1}^N \mathbbm{P}(0 < X_t^i < \varepsilon+\delta, 0 < X_t^j < \varepsilon+\delta).
\end{align*}
By employing the estimate in Proposition \ref{boundaryEst} and taking $\delta$ to zero we obtain \ref{assptBound} because $q>1$ ensures the necessary integrability.

For the exponential-tails condition consider $I = (\lambda, \eta)$ for $\eta >0$, $\lambda < \eta$. The approach for \ref{assptSpat} yields
\begin{align*}
    \mathbb{E} \left[ \int_0^T \nu_t^*(\lambda,\eta)dt \right] & \leq \liminf_{k \to \infty} \mathbb{E} \left[ \int_0^T \nu_t^{N_k}(\lambda-\delta,\eta+\delta)dt \right]\\
    &\leq \liminf_{k \to \infty} \mathbb{E} \left[ \int_0^T \nu_t^{N_k}(\lambda-\delta,\infty)dt \right]= o(e^{-\alpha (\lambda-\delta)}),
\end{align*}
by using Proposition \ref{tail}. Taking the limit $\delta \to 0$ and then $\eta \to \infty$ yields \ref{assptTail}.
\end{proof}

Combining the results from Proposition \ref{convergenceToEv} and Proposition \ref{limitRegularity} 
we can conclude our first main result Theorem \ref{existenceThm}. 

\section{Uniqueness of the SPDE} \label{Sec:LinUniq}

In this section we prove uniqueness of solutions to the SPDE in the class $\Lambda$. Before presenting the uniqueness proof we show how we can derive an alternative weak boundary formulation for the SPDE \eqref{MeasureSPDE} that is helpful in dealing with boundary terms appearing in the uniqueness proof. 

To overcome the difficulties of working with solutions of low regularity, we regularize the weak solutions to our SPDE through convolution with a nice mollifier. The elastic heat kernel $G_{\varepsilon}^{E, \kappa}(x,y)$, defined in (\ref{convDefn}), is a natural candidate for this, since it is an element of the space of test functions $\mathcal{C}^{E,\kappa}_0(\mathbb{R})$ and since its explicit form makes it easy to perform manipulations. This leads us to consider the convolution operator $T_{\varepsilon}^{E,\kappa}$ given by $(T_{\varepsilon}^{E,\kappa}\zeta)(x) = \langle \zeta , G_{\varepsilon}^{E,\kappa}(x,\cdot) \rangle$. While we use a different operator, we take the same approach as in \cite{HamblyLedger2017, HamblySojmark2019} and work in the space $H^{-1}$, the dual of the first Sobolev space. This approach now rests on the following crucial inequality, connecting $L^2$ estimates for the regularized process with an $H^{-1}$ estimate for the non-regularized  process:
\begin{equation} \label{H-1Estimate+}
    \norm{\zeta}_{-1} \leq C \liminf_{\varepsilon \to 0} \left(\norm{\partial_x^{-1}T_{\varepsilon}^{E,\kappa}\zeta}_{L^2(0,\infty)}+ \left\lvert \int_0^{\infty}T_{\varepsilon}^{E,\kappa}\zeta(y)dy \right\rvert \right),
\end{equation}
for any $\zeta \in H^{-1}$, where $\partial^{-1}_x$ denotes the anti-derivative as defined in  (\ref{antiDeriv}). The proof of this can be found in Lemma \ref{H1lemma} and further details are reserved for Appendix \ref{AppendixH1}.

Differently from \cite{HamblyLedger2017, HamblySojmark2019}, here we work with $H^{-1}$ as the dual of $H^{1}(0,\infty)$ rather than $H^{1}_0(0,\infty)$. In our proof of uniqueness, we will show that two different solutions to the SPDE, say $\nu$ and $\tilde{\nu}$, must satisfy $\mathbb{E}[\norm{\nu_t - \tilde{\nu}_t}_{-1}]=0$ for all $t\in[0,T]$. In doing so, the first term on the right-hand side of \eqref{H-1Estimate+} is only concerned with the difference $\zeta = \nu_t - \tilde{\nu}_t$ restricted to $(0,\infty)$, while forcing the second term to zero ensures that potential point masses of $\nu_t$ and $\tilde{\nu}_t$ at the origin must be of the same value (recall the solutions are sub-probability measures, which a priori need not have the same mass at any given time). As discussed in Section 2, the unique solution turns out to have a density, but we are not making any such assumptions on our class of solutions at this point.

Before proceeding to show uniqueness, we first give two auxiliary results that will help us deal with the behaviour of our estimates near the boundary.

\begin{lemma} \label{boundaryEstimate1}
Let $\nu$ be a process in the class $\Lambda$ and $p_\varepsilon$ the Gaussian heat kernel defined in \eqref{GaussianHeatKernel}. Then we have 
\begin{equation*}
    \mathbb{E} \left[ \int_0^T \int_0^{\infty} \left( \int_0^{\infty} p_{\varepsilon}(x+y)\nu_t(dy) \right)^2dxdt \right] \to 0, \qquad{\text{ as } \varepsilon \to 0}.
\end{equation*}
\end{lemma}

\begin{proof}
Arguing as in \cite[Lemma 7.6]{HamblyLedger2017}  we have
\begin{align*}
    \lvert \langle \nu_t,p_{\varepsilon}(x+\cdot) \rangle \rvert & \leq e^{-x^2/\varepsilon} \int_0^{\infty} p_{\varepsilon}(y)\nu_t(dy) \\
    & \leq c_1 e^{-x^2/\varepsilon} \varepsilon^{-1/2} [\nu_t(0,\varepsilon^\eta)+\exp(-\varepsilon^{2 \eta -1}/2)],
\end{align*}
for some $\eta \in (0,\frac{1}{2})$ and $c_1 >0$. Squaring and integrating over $x>0$ gives
\begin{equation*}
    \int_0^{\infty} \lvert \langle \nu_t,p_{\varepsilon}(x+\cdot) \rangle \rvert^2dx \leq c_2 \varepsilon^{-1/2} [\nu_t(0,\varepsilon^{\eta})^2+\exp(-\varepsilon^{2\eta-1})],
\end{equation*}
\noindent
for some $c_2 >0$. Since $\nu$ is in the class $\Lambda$, the boundary decay yields a $\gamma>0$ so that
\begin{equation*}
    \mathbbm{E} \left[ \int_0^T \int_0^{\infty} \lvert \langle \nu_t,p_{\varepsilon}(x+ \cdot) \rangle \rvert^2dx \right] = O(\varepsilon^{\eta(1+\gamma) - \frac{1}{2}}) + O(\varepsilon^{-\frac{1}{2}}\exp(-\varepsilon^{2\eta-1})).
\end{equation*}
Therefore, we get the desired estimate by choosing $\eta$ such that $(1+\gamma)^{-1}<2\eta<1$.
\end{proof}

\begin{lemma} \label{BoundaryEstimate2}
Let $\nu$ be a process in the class $\Lambda$ and let $g_{\varepsilon}^{E,\kappa}$ denote the correction term in the definition of the elastic heat kernel \eqref{ElasticKernel}. Then we have that
\begin{equation*}
    \mathbb{E} \left[ \int_0^T \int_0^{\infty} \left( \int_0^{\infty} g_{\varepsilon}^{E,\kappa}(x,y) \nu_t(dy) \right)^2 dxdt \right] \to 0, \qquad{\text{for } \varepsilon \to 0}.
\end{equation*} 
\end{lemma}

\begin{proof}
We start by looking at the inner integral. Using Lemma \ref{ElasticCorrectionTermEstimate} we get
\begin{align*}
    \int_0^{\infty} g_{\varepsilon}^{E,\kappa}(x,y) \nu_t(dy)  
    \leq  \kappa e^{-\frac{x^2}{2 \varepsilon}} \nu_t(0,\infty) 
    &\leq  \kappa e^{-\frac{x^2}{2 \varepsilon}}.
\end{align*}
Now, considering the integral in $x$ we have
\begin{align*}
    \int_0^{\infty} \left(\int_0^{\infty} g_{\varepsilon}^{E,\kappa}(x,y) \nu_t(dy)  \right)^2dx & \leq \kappa^2 \int_0^{\infty} e^{-\frac{x^2}{\varepsilon}}dx = \frac{1}{2} \kappa^2 \sqrt{\pi} \varepsilon^{1/2}.
\end{align*}
We then obtain 
\begin{equation}
    \mathbb{E}  \left[\int_0^T \int_0^{\infty} \left(\int_0^{\infty} g_{\varepsilon}^{E,\kappa}(x,y) \nu_t(dy) \right)^2 dx dt  \right]= O(\varepsilon^{1/2})
\end{equation}
and the result follows.
\end{proof}

\subsection{Weak Boundary Condition and Elastic Boundary Terms}\label{subsec:weak_boundary}

As discussed in Section \ref{sect:elastic_intro}, we need a particular weak formulation of the elastic boundary condition. This plays a critical role in dealing with the boundary terms appearing in \eqref{BoundaryTermIBP} in the final uniqueness proof, presented in the next subsection. With $\phi^{E,\kappa}_\varepsilon$ as in \eqref{elasticTestFunc}, we note that this gives a smooth function approximating 1 on $[0,\infty)$ and satisfying
\begin{align}\label{eq:phi_elastic_boundary}
\kappa \phi^{E,\kappa}_{\varepsilon}(0) = \partial_x \phi^{E,\kappa}_{\varepsilon}(0),
\end{align}
which it inherits from the elastic kernel $G^{E,\kappa}_{\varepsilon}$. We then obtain the following.
\begin{proposition}[Weak boundary condition]\label{prop:aprrox_BC} Let $(\nu,W^0)$ be a solution to \eqref{MeasureSPDE} with $\nu$ in the class $\Lambda$ and let  $\phi^{E,\kappa}_{\varepsilon}$ be as in (\ref{elasticTestFunc}). Then \eqref{eq:weak_BC} holds for all $\varepsilon>0$.

\end{proposition}

\begin{proof}
By \eqref{eq:phi_elastic_boundary}, we see that the convolution $\phi^{E,\kappa}_\varepsilon(x)=\int_0^{\infty}  G^{E,\kappa}_{\varepsilon}(x,y)dy$ results in a valid test function in $\mathcal{C}^{E,\kappa}_0(\mathbb{R})$. Using also Fubini's theorem to change the order of integration between $d\nu$ and $dx$, the assumption that $\nu$ is a solution to \eqref{MeasureSPDE} therefore gives
	\begin{align*}
	&\langle \nu_t, \phi^{E,\kappa}_{\varepsilon} \rangle -\langle \nu_0, \phi^{E,\kappa}_{\varepsilon} \rangle = \int_0^{\infty} \langle \nu_t, G^{E,\kappa}_{\varepsilon}(x,\cdot)\rangle- \langle \nu_0, G^{E,\kappa}_{\varepsilon}(x,\cdot)  \rangle dx \\
	 &\;= \int_0^{\infty} \Big( \int_0^t  \langle \nu_s,\mu_s \partial_y G_{\varepsilon}^{E,\kappa}(x,\cdot) \rangle ds  \\
	&\qquad\qquad+ \int_0^t \langle \nu_s,\frac{\sigma_s^2}{2} \partial_{yy}G_{\varepsilon}^{E,\kappa}(x,\cdot)\rangle ds + \int_0^t \rho_s \langle \nu_s,\sigma_s \partial_y G_{\varepsilon}^{E,\kappa}(x,\cdot)\rangle dW_s^0 \Big)dx.
	\end{align*}
	Next, we employ Lemma \ref{derivativeSwitch} to switch the derivatives from $y$ to $x$. This yields
	\begin{align*}
	\langle \nu_t,& \phi^{E,\kappa}_{\varepsilon} \rangle -\langle \nu_0, \phi^{E,\kappa}_{\varepsilon} \rangle\\
	= &\int_0^{\infty} \Big(\int_0^t- \partial_x \langle \nu_s,\mu_s G_{\varepsilon}^E(x,\cdot) \rangle ds + \int_0^t  \partial_{xx} \langle \nu_s,\frac{\sigma_s^2}{2} G_{\varepsilon}^E(x,\cdot) \rangle ds \\
	&- \int_0^t \rho_s  \partial_x \langle \nu_s,\sigma_s G_{\varepsilon}^E(x,\cdot) \rangle dW_s^0 \Big)dx \\
	&+ \int_0^{\infty}\left(2 \int_0^t  \partial_x \langle \nu_s,\mu_s p_{\varepsilon}(x+\cdot) \rangle ds + 2 \int_0^t \rho_s  \partial_x \langle \nu_s,\sigma_s p_{\varepsilon}(x+\cdot)\rangle dW_s^0 \right)dx \\
	&- \int_0^{\infty}\left(2 \int_0^t  \partial_x \langle \nu_s,\mu_s g_{\varepsilon}^{E,\kappa}(x,\cdot)\rangle ds - 2 \int_0^t \rho_s  \partial_x \langle \nu_s,\sigma_s g_{\varepsilon}^{E,\kappa}(x,\cdot) \rangle dW_s^0 \right)dx.
	\end{align*}
	with $g_{\varepsilon}^{E,\kappa}$ as defined in \eqref{ElasticKernel} and $p_{\varepsilon}$ the Gaussian heat kernel given in (\ref{GaussianHeatKernel}).
	We can now interchange the order of integration again and compute the integrals in $x$ to obtain
	\begin{align*}
	\langle \nu_t, \phi^{E,\kappa}_{\varepsilon} \rangle -\langle \nu_0, \phi^{E,\kappa}_{\varepsilon} \rangle  = &\int_0^t  \langle \nu_s,\mu_s G_{\varepsilon}^{E,\kappa}(0,\cdot)\rangle ds - \int_0^t  \partial_x \langle \nu_s,\frac{\sigma^2}{2} G_{\varepsilon}^{E,\kappa}(0,\cdot)\rangle ds \\
	&+ \int_0^t \rho_s  \langle \nu_s,\sigma_s G_{\varepsilon}^{E,\kappa}(0,\cdot)\rangle dW_s^0 \\
	&- 2 \int_0^t \langle \nu_s,\mu_sp_{\varepsilon}(\cdot)\rangle ds - 2 \int_0^t \rho_s \langle \nu_s,\sigma_s p_{\varepsilon}(\cdot)\rangle dW_s^0 \\
	&+ 2 \int_0^t \langle \nu_s,\mu_s g_{\varepsilon}^{E,\kappa}(0,\cdot)\rangle ds + 2\int_0^t  \rho_s \langle \nu_s,\sigma_sg_{\varepsilon}^{E,\kappa}(0,\cdot)\rangle dW_s^0. 
	\end{align*}
	Using the specific form of the elastic heat kernel we have that
	\begin{align*}
	\int_0^t \langle \nu_s,\mu_s G_{\varepsilon}^{E,\kappa}(0,\cdot)\rangle ds =& 2\int_0^t \langle \nu_s,\mu_s p_{\varepsilon}(\cdot)\rangle ds - \int_0^t \langle \nu_s,g_{\varepsilon}^{E,\kappa}(0,\cdot)\rangle ds \\
	\int_0^t \rho_s \langle \nu_s,\sigma_s G_{\varepsilon}^{E,\kappa}(0,\cdot)\rangle dW_s^0 = &2 \int_0^t \rho_s \langle \nu_s,\sigma_s p_{\varepsilon}(\cdot)\rangle dW_s^0- \int_0^t \rho_s \langle \nu_s,\sigma_s g_{\varepsilon}^{E,\kappa}(0,\cdot)\rangle dW_s^0.
	\end{align*}
	By using these relations the equation above simplifies to
	\begin{align*}
	\langle \nu_t, \phi^{E,\kappa}_{\varepsilon} \rangle -\langle \nu_0, \phi^{E,\kappa}_{\varepsilon} \rangle  = &-  \int_0^t \partial_x \langle \nu_s,\frac{\sigma_s^2}{2}G_{\varepsilon}^{E,\kappa}(0,\cdot)\rangle ds \\
	&+  \int_0^t \langle \nu_s,\mu_s g_{\varepsilon}^{E,\kappa}(0,\cdot)\rangle ds + \int_0^t  \rho \langle \nu_s,\sigma_s g_{\varepsilon}^{E,\kappa}(0,\cdot)\rangle dW_s^0.
	\end{align*}
	Finally, we can then use the boundary condition of the elastic heat kernel to obtain the relaxed boundary condition \eqref{eq:weak_BC}.
\end{proof}

\begin{lemma} [Elastic Boundary Term] \label{IBPBoundary}
Let $\nu$ and $\tilde{\nu}$ be two solutions to the SPDE \eqref{MeasureSPDE} in the class $\Lambda$ and denote their difference by $\Delta$, i.e. $\Delta \coloneqq \nu - \tilde{\nu}$. Then we have
	\begin{align*}
	\mathbb{E} \Bigg[ \int_0^t - \partial_x^{-1}T_{\varepsilon}^{E,\kappa} \Delta_s(0) \sigma_s^2(0) T_{\varepsilon}^{E,\kappa} \Delta_s(0)  &- \partial_x^{-1} T_{\varepsilon}^{E,\kappa} \Delta_s(0) \left(\mathcal{E}_{s,\varepsilon}^{\sigma^2}- \tilde{\mathcal{E}}_{s,\varepsilon}^{\sigma^2} \right)(0)ds \Bigg] \\
	&= - \frac{1}{\kappa} \mathbb{E}\left[ (\partial_x^{-1}T_{\varepsilon}^{E,\kappa} \Delta_t(0))^2 \right] + o(1).
	\end{align*}
with $\mathcal{E}^{\sigma^2}_{s,\varepsilon}, \tilde{\mathcal{E}}^{\sigma^2}_{s,\varepsilon}$ error terms as defined in \eqref{errorTermDefn} for measures $\nu_s$ and $\tilde{\nu_s}$, respectively.
\end{lemma}

\begin{proof}
	Combining Proposition \ref{prop:aprrox_BC} with the definition of the operator $T_{\varepsilon}^{E,\kappa}$, the anti-derivative and the error terms defined in (\ref{errorTermDefn}) we obtain
	\begin{align} \label{SemiMartEq}
	\begin{split}
	&\int_0^s \sigma_u^2(0) T_{\varepsilon}^{E,\kappa} \Delta_u(0)du = \frac{2}{\kappa} \partial_x^{-1} T_{\varepsilon}^{E,\kappa} \Delta_s(0) - \int_0^s (\mathcal{E}_{u,\varepsilon}^{\sigma^2}- \mathcal{E}_{u,\varepsilon}^{\Tilde{\sigma}^2})(0)du \\
	&\qquad+ \frac{2}{\kappa}  \int_0^s \langle \Delta_u,\mu_u g_{\varepsilon}^{E,\kappa}(0,\cdot)\rangle du+ \frac{2}{\kappa}  \int_0^s \rho_u \langle \Delta_u,\sigma_u g_{\varepsilon}^{E,\kappa}(0,\cdot) \rangle dW_u^0.
	\end{split}
	\end{align}
	With this relation we obtain
	\begin{align*}
	-\mathbb{E} \Bigg[\int_0^t \sigma_s(0)^2 \partial_x^{-1}T_{\varepsilon}^{E,\kappa} &\Delta_s(0) T_{\varepsilon}^{E,\kappa} \Delta_s(0) ds \Bigg] \\
	= &- \mathbb{E} \left[\int_0^t \partial_x^{-1} T_{\varepsilon}^{E,\kappa} \Delta_s(0) d(\int_0^{s}\sigma_u^2(0) T_{\varepsilon}^{E,\kappa} \Delta_u(0)du)  \right] \\
	= &- \frac{1}{\kappa} \mathbb{E} \left[2 \int_0^t \partial_x^{-1}T_{\varepsilon}^{E,\kappa} \Delta_s(0) d(\partial_x^{-1}T_{\varepsilon}^{E,\kappa} \Delta_s(0)) \right] \\
	&+  \mathbb{E} \left[\int_0^t\partial_x^{-1}T_{\varepsilon}^{E,\kappa} \Delta_s(0)(\mathcal{E}_{s,\varepsilon}^{\sigma^2}- \tilde{\mathcal{E}}_{s,\varepsilon}^{\sigma^2})(0)ds \right] \\
	&- \frac{2}{\kappa} \mathbb{E} \left[\int_0^t \partial_x^{-1}T_{\varepsilon}^{E,\kappa}\Delta_s(0) \langle \Delta_s,\mu_s g_{\varepsilon}^{E,\kappa}(0,\cdot)\rangle ds \right] \\
	&- \frac{2}{\kappa} \mathbb{E} \left[\int_0^t \partial_x^{-1}T_{\varepsilon}^{E,\kappa} \Delta_s(0) \rho_s \langle \Delta_s,\sigma_s g_{\varepsilon}^{E,\kappa}(0,\cdot)\rangle dW_s^0 \right].
	\end{align*}
	
	By combining the boundedness of $\sigma$ with Lemma \ref{ElasticCorrectionTermEstimate} we get that the stochastic integral in the last line is a martingale and thus the term vanishes. For the first term, using the fact that $\partial_x^{-1}T_{\varepsilon}^{E,\kappa}\Delta_{\cdot}(0)$ is a continuous semimartingale, we have
	\begin{equation*}
	(\partial_x^{-1}T_{\varepsilon}^{E,\kappa}\Delta_t(0))^2 = 2 \int_0^t \partial_x^{-1}T_{\varepsilon}^{E,\kappa} \Delta_s(0)d(\partial_x^{-1}T_{\varepsilon}^{E,\kappa} \Delta_s(0)) + [\partial_x^{-1}T_{\varepsilon}^{E,\kappa} \Delta_{\cdot}(0)]_t.
	\end{equation*}
	Using equation (\ref{SemiMartEq}) we obtain that
	\begin{equation*}
	[\partial_x^{-1}T_{\varepsilon}^{E,\kappa} \Delta_{\cdot}(0)]_t =  \int_0^t \rho_s^2  (\langle \Delta_s,\sigma_s g_{\varepsilon}^{E,\kappa}(0,\cdot)\rangle )^2ds.
	\end{equation*}
	This yields
	\begin{align}\label{J1J2Equation}
	&- \mathbb{E} \Bigg[\int_0^t \sigma_s(0)^2 \partial_x^{-1}T_{\varepsilon}^{E,\kappa} \Delta_s(0) T_{\varepsilon}^{E,\kappa} \Delta_s(0) ds \Bigg]\\
	&= - \frac{1}{\kappa} \mathbb{E}\left[ (\partial_x^{-1}T_{\varepsilon}^{E,\kappa} \Delta_t(0))^2 \right] 
	+  \mathbb{E} \left[\int_0^t\partial_x^{-1}T_{\varepsilon}^{E,\kappa} \Delta_s(0)(\mathcal{E}_{s,\varepsilon}^{\sigma^2}- \tilde{\mathcal{E}}_{s,\varepsilon}^{\sigma^2})(0)ds \right] 
	+ J_1 + J_2
	\end{align}
	with the two terms $J_1$ and $J_2$ defined as
	\begin{align*}
	J_1 &\coloneqq  \frac{1}{\kappa} \mathbb{E} \left[\int_0^t \rho_s^2 (\langle \Delta_s,\sigma_s g_{\varepsilon}^{E,\kappa}(0,\cdot)\rangle )^2 ds  \right] \\
	J_2 &\coloneqq  - \frac{2}{\kappa} \mathbb{E} \left[\int_0^t \partial_x^{-1}T_{\varepsilon}^{E,\kappa}\Delta_s(0) \langle \Delta_s,\mu_s g_{\varepsilon}^{E,\kappa}(0,\cdot)\rangle ds \right].
	\end{align*}
	
	From Lemma \ref{ElasticCorrectionTermEstimate} we have the estimate
	\begin{equation*}
	g_{\varepsilon}^{E,\kappa}(0,y) \leq \kappa e^{-\frac{y^2}{2 \varepsilon}}. 
	\end{equation*}
	For a process $\xi$ in the class $\Lambda$ we get
	\begin{align}\label{ElasticCorrectionLambdaEst}
	\begin{split}
	\int_0^\infty g_{\varepsilon}^{E,\kappa}(0,y) \xi_s(dy) &\leq \kappa \int_0^{\varepsilon^{1/4}} e^{-\frac{y^2}{2 \varepsilon}} \xi_s(dy) + \kappa \int_{\varepsilon^{1/4}}^\infty e^{-\frac{y^2}{2 \varepsilon}} \xi_s(dy) \\
	&\leq \kappa \xi_s(0,\varepsilon^{1/4}) + \kappa e^{-\frac{\varepsilon^{-1/2}}{2}}
	\end{split}
	\end{align}
	The coefficients $\mu, \sigma, \rho$ are bounded by assumption and
	\begin{equation*}
	\lvert \partial_x^{-1}T_{\varepsilon}^{E,\kappa}\Delta_s(0) \rvert \leq \lvert \Delta_s(0,\infty) \rvert \leq 2.   
	\end{equation*}
	We can combine this with \eqref{ElasticCorrectionLambdaEst} to estimate
	\begin{align} \label{J1Convergence}
	J_1 \leq C \left( \int_0^t  \mathbb{E} \left[ \nu_s(0,\varepsilon^{1/4})^2 \right] ds + \int_0^t  \mathbb{E}  \left[\tilde{\nu}_s(0,\varepsilon^{1/4})^2 \right] ds  \right) + O(e^{-\varepsilon^{-1/2}}) \to 0,
	\end{align}
	as $\varepsilon \to 0$, because $\nu$ and $\tilde{\nu}$ are in the class $\Lambda$. For $J_2$ we similarly get
	\begin{align} \label{J2Convergence}
	\begin{split}
	\lvert J_2 \rvert  &\leq C \mathbb{E} \left[ \int_0^t \langle \nu_s, g_\varepsilon^{E,\kappa}(0,\cdot) + \langle \tilde{\nu}_s, g_\varepsilon^{E,\kappa}(0,\cdot) \rangle \rangle ds \right] \\
	&\leq \int_0^t \mathbb{E} \left[\nu_s(0,\varepsilon^{1/4})\right] + \mathbb{E}\left[\tilde{\nu}_s(0,\varepsilon^{1/4})\right]ds + O(e^{-\varepsilon^{-1/2}/2}) \to 0, \quad \text{as $\varepsilon \to 0$},
	\end{split}
	\end{align}
	again, by the properties of the class $\Lambda$. Combining \eqref{J1J2Equation}, \eqref{J1Convergence} and \eqref{J2Convergence} we obtain
	
	\begin{align*}
	&- \mathbb{E} \Bigg[\int_0^t \sigma_s(0)^2 \partial_x^{-1}T_{\varepsilon}^{E,\kappa} \Delta_s(0) T_{\varepsilon}^{E,\kappa} \Delta_s(0) ds \Bigg]\\
	&= - \frac{1}{\kappa} \mathbb{E}\left[ (\partial_x^{-1}T_{\varepsilon}^{E,\kappa} \Delta_t(0))^2 \right] 
	+  \mathbb{E} \left[\int_0^t\partial_x^{-1}T_{\varepsilon}^{E,\kappa} \Delta_s(0)(\mathcal{E}_{s,\varepsilon}^{\sigma^2}- \tilde{\mathcal{E}}_{s,\varepsilon}^{\sigma^2})(0)ds \right]
	+ o(1)
	\end{align*}
	as $\varepsilon \to 0$, and so the result follows.
\end{proof}

\subsection{Proof of uniqueness for solutions to the SPDE}\label{subsec:uniqueness_proof}

We can now give the proof of Theorem \ref{uniqueness}. Lemma \ref{IBPBoundary} from the previous section plays an important role in our arguments. It is used in \eqref{BoundaryTermIBP} to deal with the boundary terms that arise when integrating by parts.

\begin{proof}[Proof of Theorem \ref{uniqueness}]
Let $(\nu,W^0)$ be a solution to the SPDE (\ref{MeasureSPDE}) in the class $\Lambda$. 
Take the elastic heat kernel function $y \mapsto G_{\varepsilon}^{E,\kappa}(x,y)$ as a test function in the SPDE and apply the switching of derivatives from $y$ to $x$ using Lemma \ref{derivativeSwitch} to obtain the following dynamics for the smoothed measure $T^{E,\kappa}_{\varepsilon} \nu_t$ 
\begin{align*}
    dT_{\varepsilon}^{E,\kappa} \nu_t(x) = &- \partial_x \langle \nu_t,\mu_t G_{\varepsilon}^{E,\kappa}(x,\cdot) \rangle dt + \frac{1}{2} \partial_{xx} \langle \nu_t,\sigma_t^2 G_{\varepsilon}^{E,\kappa}(x,\cdot) \rangle dt \\
    &-  \rho_t \partial_x  \langle \nu_t,\sigma_t G_{\varepsilon}^{E,\kappa}(x,\cdot) \rangle dW_t^0 \\
    &+ 2 \partial_x \langle \nu_t, \mu_t p_{\varepsilon}(x+\cdot) \rangle dt + 2 \rho_t \partial_x \langle  \nu_t,\sigma_t p_{\varepsilon}(x+ \cdot) \rangle dW_t^0 \\
    &-2 \partial_x \langle \nu_t,\mu_t g_{\varepsilon}^{E,\kappa}(x,\cdot) \rangle dt - 2 \rho_t \partial_x \langle \nu_t,\sigma_t g_{\varepsilon}^{E,\kappa}(x, \cdot) \rangle dW_t^0.
\end{align*}
Next, we integrate to introduce the anti-derivative defined in (\ref{antiDeriv})
\begin{align*}
    d\partial_x^{-1}T_{\varepsilon}^{E,\kappa} \nu_t(x) = &- \langle \nu_t,\mu_t G_{\varepsilon}^{E,\kappa}(x,\cdot) \rangle dt + \frac{1}{2} \partial_{x} \langle \nu_t,\sigma_t^2 G_{\varepsilon}^{E,\kappa}(x,\cdot)\rangle dt \\
    &- \rho_t  \langle \nu_t,\sigma_t G_{\varepsilon}^{E,\kappa}(x,\cdot)\rangle dW_t^0 \\
    &+ 2 \langle \nu_t,\mu_t p_{\varepsilon}(x+\cdot)\rangle dt + 2 \rho_t  \langle \nu_t,\sigma_t p_{\varepsilon}(x+ \cdot)\rangle dW_t^0 \\
    &-2  \langle \nu_t,\mu_t g_{\varepsilon}^{E,\kappa}(x,\cdot) \rangle dt - 2 \rho_t  \langle \nu_t, \sigma_t g_{\varepsilon}^{E,\kappa}(x, \cdot) \rangle dW_t^0.
\end{align*}
We can take the coefficients $\mu, \sigma$ and $\sigma^2$ outside of the integration against the measure $\nu_t$ by introducing error terms $\mathcal{E}_{t,\varepsilon}^h$ defined by
\begin{equation} \label{errorTermDefn}
    \mathcal{E}_{t,\varepsilon}^h \coloneqq \langle \nu_t,h_t(\cdot)G^{E,\kappa}_{\varepsilon}(x,\cdot)\rangle -h_t(x)T_{\varepsilon}^{E,\kappa}\nu_t(x).
\end{equation}
for $h_t(x) = h(t,x)$ representing $\mu,\sigma$ or $\sigma^2$. The equation then becomes
\begin{align*}
    d\partial_x^{-1}T_{\varepsilon}^{E,\kappa} \nu_t  = &-(\mu_t T_{\varepsilon}^{E,\kappa} \nu_t + \mathcal{E}_{t,\varepsilon}^{\mu}) dt + \frac{1}{2} \partial_x(\sigma^2_t T_{\varepsilon}^{E,\kappa}\nu_t+\mathcal{E}_{t,\varepsilon}^{\sigma^2})dt \\
    &- \rho_t (\sigma_t T_{\varepsilon}^{E,\kappa} \nu_t + \mathcal{E}_{t,\varepsilon}^{\sigma})dW_t^0 \\
    &+ 2  \langle \nu_t,\mu_t p_{\varepsilon}(x+\cdot) \rangle dt + 2 \rho_t  \langle \nu_t,\sigma_t p_{\varepsilon}(x+ \cdot)\rangle dW_t^0 \\
    &-2 \langle \nu_t,\mu_t g_{\varepsilon}^{E,\kappa}(x+\cdot) \rangle dt - 2 \rho_t \langle \nu_t, \sigma_t g_{\varepsilon}^{E,\kappa}(x+ \cdot)\rangle dW_t^0.
\end{align*}
To simplify the notation, we denote by $o_{sq}(1)$ any family of functions $\{(f_{t,\varepsilon})_{t \in [0,T]} \}_{\varepsilon>0}$ such that
\begin{equation}
    \mathbb{E} \int_0^T \norm{f_{t,\varepsilon}}^2_{L^2(0,\infty)}dt \to 0, \quad \text{as } \varepsilon \to 0.
\end{equation}
By using the limit behavior from Lemma \ref{boundaryEstimate1} and Lemma \ref{BoundaryEstimate2}, as well as the results for $\mathcal{E}^h_{t,\varepsilon}$ from \cite[Lemma 8.1]{HamblyLedger2017}, which also hold in the elastic case because of the regularity of the class $\Lambda$,  we have 
\begin{align} \label{dynamicsForYoung}
\begin{split}
    d\partial_x^{-1}T_{\varepsilon}^{E,\kappa} \nu_t  = &-\mu_t T_{\varepsilon}^{E,\kappa} \nu_t  dt + \frac{1}{2} \partial_x(\sigma^2_t T_{\varepsilon}^{E,\kappa}\nu_t+\mathcal{E}_{t,\varepsilon}^{\sigma^2})dt \\
    &- \rho_t \sigma_t T_{\varepsilon}^{E,\kappa} \nu_t dW_t^0 
    + o_{sq}(1)dt + o_{sq}(1)dW_t^0.
\end{split}
\end{align}
Let $(\tilde{\nu},W^0)$ be another solution to the linear SPDE with the same Brownian motion $W^0$ and define the difference of the two solutions as $\Delta \coloneqq \nu- \tilde{\nu} $. The dynamics of this are
\begin{align*}
    d\partial_x^{-1}T_{\varepsilon}^{E,\kappa} \Delta_t  = &-\mu_t T_{\varepsilon}^{E,\kappa} \Delta_t  dt + \frac{1}{2} \partial_x(\sigma_t^2 T_{\varepsilon}^{E,\kappa} \Delta_t + \mathcal{E}_{t,\varepsilon}^{\sigma^2}-\tilde{\mathcal{E}} _{t,\varepsilon}^{\sigma^2})dt \\
    &- \rho_t \sigma_t T_{\varepsilon}^{E,\kappa} \Delta_t dW_t^0 
    + o_{sq}(1)dt + o_{sq}(1)dW_t^0.
\end{align*}

Next we use It{\^o}'s formula to find an equation for the square of the anti-derivative
\begin{align*}
    d(\partial_x^{-1} T^{E,\kappa}_{\varepsilon} \Delta_t)^2  = &- 2\mu_t(\partial_x^{-1}T^{E,\kappa}_{\varepsilon} \Delta_t) T_{\varepsilon}^{E,\kappa} \Delta_t dt \\
    &+ (\partial_x^{-1}T^{E,\kappa}_{\varepsilon} \Delta_t) \partial_x (\sigma_t^2 T_{\varepsilon}^{E,\kappa} \Delta_t + \mathcal{E}_{t,\varepsilon}^{\sigma^2}- \tilde{\mathcal{E}}_{t,\varepsilon}^{\sigma^2})dt \\
    &+ \sigma^2_t \rho_t^2 (T^{E,\kappa}_{\varepsilon}\Delta_t)^2dt \\
    &- 2 (\partial_x^{-1}T^{E,\kappa}_{\varepsilon}\Delta_t) \sigma_t \rho_t T^{E,\kappa}_{\varepsilon}\Delta_t dW_t^0 \\
    &+ \rho_t \sigma_t T_{\varepsilon}^{E,\kappa} \Delta_t o_{sq}(1)dt \\
    &+ (\partial_x^{-1}T^{E,\kappa}_{\varepsilon} \Delta_t)o_{sq}(1) dt\\
    &+ (\partial_x^{-1}T^{E,\kappa}_{\varepsilon}\Delta_t)o_{sq}(1) dW_t + o_{sq}(1)^2dt.
\end{align*}
We are interested in estimates for $\mathbb{E} [ \Vert \partial_x^{-1}T_{\varepsilon}^{E,\kappa} \Delta_t\Vert_2^2 ]$ to use the $H^{-1}$-estimate (\ref{H-1Estimate}). We obtain these by integrating over $x>0$ and taking expectation.  We deal with the terms on the right-hand side individually. Combining the subexponential-tails property of $\nu_t$ and $\tilde{\nu}_t$ with \cite[Lemma 8.3 and Lemma 8.7]{HamblyLedger2017}, which are easily adapted to the elastic heat kernel, the stochastic integral terms are martingales and thus vanish when we take expectation. For the first term we use the  generalized Young's inequality with free parameter $\eta>0$ to obtain
\begin{align} \label{DriftYoungIneq}
\begin{split}
    \mathbb{E} \Bigg[-2 &\int_0^t \int_0^{\infty} \mu_s(\partial_x^{-1}T^{E,\kappa}_{\varepsilon} \Delta_s) T_{\varepsilon}^{E,\kappa} \Delta_s dxds  \Bigg] \\ &\leq c_{\eta} \int_0^t \mathbb{E} \left[ \norm{\partial_x^{-1}T_{\varepsilon}^{E,\kappa} \Delta_s}_2^2 \right]ds
    +  \eta \mathbb{E} \left[ \int_0^t \int_0^{\infty} (\mu_s)^2 \lvert T_{\varepsilon}^{E,\kappa} \Delta_s \rvert^2 dxds\right].
\end{split}
\end{align}
We employ integration by parts followed by the generalized Young's inequality for the second term. Lemma \ref{IBPBoundary} from Section \ref{subsec:weak_boundary} allows us to handle the boundary term that appears when integrating by parts. Thus, we have
\begin{align} \label{BoundaryTermIBP}
\begin{split}
    \mathbb{E} &\Bigg[ \int_0^t \int_0^{\infty} (\partial_x^{-1}T^{E,\kappa}_{\varepsilon} \Delta_s) \partial_x (\sigma_s^2 T^{E,\kappa}_{\varepsilon} \Delta_s + \mathcal{E}_{s,\varepsilon}^{\sigma^2}- \tilde{\mathcal{E}}_{s,\varepsilon}^{\sigma^2})dx ds \Bigg]  \\
     & \leq - \frac{1}{\kappa} \mathbb{E}\left[ (\partial_x^{-1}T_{\varepsilon}^{E,\kappa} \Delta_s(0))^2 \right]  - \mathbb{E} \left[ \int_0^t \int_0^{\infty} \sigma_s^2 \lvert T^{E,\kappa}_{\varepsilon} \Delta_s \rvert^2 dx ds \right] \\ &\qquad + \eta \mathbb{E} \left[\int_0^t \int_0^{\infty} \lvert  T_{\varepsilon}^{E,\kappa} \Delta_s \rvert ^2 dxds \right] + o(1).
\end{split}
\end{align}
We leave the third term as it is. For the remaining terms we proceed to use the generalized Young's inequality as in \eqref{DriftYoungIneq}. Putting everything back together we obtain
\begin{align*}
   & \mathbb{E} \Bigg[\norm{\partial_x^{-1}T_{\varepsilon}^{E,\kappa} \Delta_t}^2_2 \Bigg] + \frac{1}{\kappa} \mathbb{E}\left[ (\partial_x^{-1}T_{\varepsilon}^{E,\kappa} \Delta_t(0))^2 \right] 
    \leq c_{\eta} \mathbb{E} \left[ \int_0^t \norm{\partial_x^{-1}T^{E,\kappa}_{\varepsilon} \Delta}_2^2ds \right] \\
    &\qquad - \mathbb{E} \left[\int_0^t \int_0^{\infty} (\sigma_s^2 - \sigma_s^2 \rho_s^2(1+\eta) - \mu_s^2 \eta - \eta ) \lvert T^{E,\kappa}_{\varepsilon} \Delta_s \rvert^2 dx ds \right]  + o(1).
\end{align*}
Since $\rho_s^2 < 1$ we can choose the free parameter $\eta$ small enough such that
\begin{equation}
    \sigma_s^2 - \sigma_s^2 \rho_s^2(1+\eta) - \mu_s^2 \eta - \eta \geq c_0
\end{equation}
for all $x$ and $s$ and for some constant $c_0 >0$. This yields
\begin{align} \label{GronwallSetup}
\begin{split}
   &  \mathbb{E} \left[\norm{\partial_x^{-1}T_{\varepsilon}^{E,\kappa} \Delta}^2_2 \right] + \frac{1}{\kappa} \mathbb{E}\left[ (\partial_x^{-1}T_{\varepsilon}^{E,\kappa} \Delta_t(0))^2 \right] \\
    &\qquad \leq c_{\eta}  \int_0^t\mathbb{E} \left[ \norm{\partial_x^{-1}T^{E,\kappa}_{\varepsilon} \Delta_s}_2^2 \right]ds-c_0 \mathbb{E} \left[\int_0^t \norm{T_\varepsilon^{E,\kappa} \Delta_s}_2^2 ds\right] + o(1).
\end{split}    
\end{align}
We then get the estimate 
\begin{align*}
\begin{split}
    &\mathbb{E} \left[\norm{\partial_x^{-1}T_{\varepsilon}^{E,\kappa} \Delta}^2_2 \right] + \frac{1}{\kappa} \mathbb{E}\left[ (\partial_x^{-1}T_{\varepsilon}^{E,\kappa} \Delta_t(0))^2 \right] \\
    &\qquad \leq c_{\eta}  \int_0^t\mathbb{E} \left[ \norm{\partial_x^{-1}T^{E,\kappa}_{\varepsilon} \Delta_s}_2^2 \right]ds+ c_\eta \int_0^t \frac{1}{\kappa} \mathbb{E}\left[ (\partial_x^{-1}T_{\varepsilon}^{E,\kappa} \Delta_s(0))^2 \right]ds   + o(1).
    \end{split}
\end{align*}
By Gronwall's lemma it follows that
\begin{equation} \label{sumOfSq}
    \lim_{\varepsilon \to 0} \left(\mathbb{E} \left[\norm{\partial_x^{-1}T_{\varepsilon}^{E,\kappa} \Delta_t}^2_2 \right] + \frac{1}{\kappa} \mathbb{E}\left[ (\partial_x^{-1}T_{\varepsilon}^{E,\kappa} \Delta_t(0))^2 \right] \right) = 0.
\end{equation}
Now, recall  the $H^{-1}$-estimate (\ref{H-1Estimate}). We can find another constant $C_1$ such that
\begin{equation}\label{eq:H^-1}
    \norm{\Delta_{t}}_{-1} \leq C_1  \liminf_{\varepsilon \to 0} \left(\norm{\partial_x^{-1}T^{E,\kappa}_{\varepsilon}\Delta_{t}}_{L^2(0,\infty)}+ \frac{1}{\sqrt{\kappa}} \left\lvert \int_0^{\infty}T^{E,\kappa}_{\varepsilon}\Delta_{t}(y)dy \right\rvert \right).
\end{equation}
Using $(a+b)^2 \leq2 ( a^2 + b^2)$ and Fatou's lemma, we obtain from (\ref{sumOfSq}) that
\begin{align*}
    \mathbb{E} \left[\liminf_{\varepsilon \to 0} \left(\norm{\partial_x^{-1}T^E_{\varepsilon}\Delta_{t }}_{L^2(0,\infty)}+ \frac{1}{\sqrt{\kappa}} \left\lvert \int_0^{\infty}T^E_{\varepsilon}\Delta_{t}(y)dy \right\rvert \right)^2 \right] = 0
\end{align*}
and so, as a consequence of \eqref{eq:H^-1}, $\mathbb{E} \left[\norm{\Delta_{t}}_{-1} \right] = 0$. In turn, we can deduce the equality $\nu_t = \Tilde{\nu}_t$ for all $t \in [0,T]$, with probability 1, which concludes the proof.
\end{proof}

\section{Existence of a density and its integrability} \label{DensityRegularitySection}

In this section we establish the existence and regularity results for the densities of the sub-probability measures $\nu_t, t \in [0,T]$, for a solution to the SPDE \eqref{MeasureSPDE}. We use an approach based on energy estimates in $L^2$, similar to the arguments of Section \ref{Sec:LinUniq}, to establish existence of a density of $\nu_t$ restricted $(0,\infty)$. Unfortunately, these arguments do not allow us to rule out the presence of an atom at the boundary. However, as we see in Section 6.1, this can be achieved by a result of \cite{Burdzy2003} for reflected Brownian motion.

 \begin{lemma} \label{IntegratedSmoothedL2}
Let $(\nu,W^0)$ be the unique solution to the SPDE \eqref{MeasureSPDE} with $\nu$ in the class $\Lambda$. Then we have
\begin{equation*}
      \lim_{\varepsilon \to 0} \mathbb{E } \left[\int_0^t  \norm{T_\varepsilon^{E,\kappa} \nu_s}_2^2 ds \right] < \infty\qquad \text{for all} \quad t\in[0,T].
\end{equation*}
\end{lemma}
\begin{proof}
The proof follows by performing the same estimates as in the proof of Theorem \ref{uniqueness} but for $\nu$ instead of the difference of two solutions. Instead of disregarding the negative term on the right-hand side involving $c_0$ as in \eqref{GronwallSetup}, we move it over to the left-hand side and suitably adjust the application of Gronwall's lemma.
\end{proof}
Using this result we deduce that there is an $L^2$-density for a.e.~$(\omega,t) \in \Omega \times [0,T]$.

\begin{proposition}\label{PropIntegratedDensityRegularity}
Let $(\nu,W^0)$ be the unique solution to the SPDE \eqref{MeasureSPDE} with $\nu$ in the class $\Lambda$. Then for almost every $(\omega,t) \in \Omega \times [0,T]$ the measure $\nu_t(\omega)$ has a density $V_t(\omega)$ on $[0,\infty)$ and we have the integrated estimate
\begin{equation}
\mathbb{E} \left[ \int_0^t \norm{V_s}_2^2 ds\right] < \infty,\qquad \text{for all} \quad t \in [0,T].
\end{equation}
\end{proposition}
\begin{proof}
 Fatou's lemma applied to the result of Lemma \ref{IntegratedSmoothedL2}  yields a  bounded sequence $(T_{\varepsilon_n}^{E,\kappa}\nu_t)_n$ in $L^2$ for almost every $(\omega, t) \in \Omega \times [0,T]$. Combining the weak compactness of bounded sequences in a Hilbert space with the weak convergence of the smoothed measure, along with Schwartz functions $\mathcal{S}$ being dense in $L^2$, gives the result.
\end{proof}

Fixing $\omega$ (in an almost sure set), the above proposition only guarantees that $V_t(\omega)$ exists in $L^2(0,\infty)$ for almost every $t\in[0,T]$. Restricting each measure $\nu_t$ to $(0,\infty)$, we obtain instead the existence of a density for all $t\in[0,T]$, with probability 1, but we only show $L^2$ integrability away from the origin, which of course leaves the possibility of a Dirac mass at zero. Before verifying this, we first show several auxiliary results. 

Take sets $U \Subset W \Subset (0,\infty)$, where $\Subset$ denotes compact containment, and let $\psi$ be a smooth cut-off function such that $\psi = 1$ on $U$, $\psi \in (0,1)$ on $W \setminus U$ and $\psi=0$ otherwise. We have
\begin{equation*}
    \lvert \partial_x \psi \rvert + \lvert \partial_{xx} \psi \rvert \leq C 1_{W \setminus U}.
\end{equation*}
where the constant $C$ only depends on $U$ and $W$.
Moreover, we define the error term
\begin{equation} \label{DerivativeError1}
    \Bar{\mathcal{E}}_{t,\varepsilon}^h(x) = \partial_x \nu_t(h_t G_\varepsilon^{E,\kappa}(x,\cdot))-h_t(x) \partial_x T_\varepsilon^{E,\kappa} \nu_t(x) + \partial_x h_t(x) \mathcal{H}_{t,\varepsilon}(x),
\end{equation}
where
\begin{equation} \label{DerivativeError2}
    \mathcal{H}_{t,\varepsilon}(x) = \nu_t((x-y) \partial_x G_\varepsilon^{E,\kappa}(x,\cdot)),
\end{equation}
for a function $h$ representing $\mu, \sigma^2$ or $\sigma$.
\begin{lemma} \label{BoundaryDerivateLemma1}
Let $(\nu,W^0)$ be the unique solution to the SPDE \eqref{MeasureSPDE} with $\nu$ in the class $\Lambda$. Then we have the convergence
\begin{equation*}
    \mathbb{E} \left[ \int_0^t \int_0^\infty \psi^2(x) \left(\int_0^\infty \partial_x p_\varepsilon(x+y)\nu_s(dy) \right)^2dxds \right] \to 0, \quad \text{as $\varepsilon \to 0$}
\end{equation*}
with the cut-off function $\psi$ specified above.
\end{lemma}

\begin{proof}
Splitting the derivative of the Gaussian heat kernel into two parts we get
\begin{align*}
    &\int_0^\infty \psi^2(x) \left(\int_0^\infty \frac{1}{\sqrt{2 \pi \varepsilon}} \frac{x+y}{\varepsilon} e^{-\frac{(x+y)^2}{2\varepsilon}} \nu_s(dy) \right)^2dx \\
    &\leq 2 \int_0^\infty \psi^2(x) e^{-x^2/ \varepsilon}  \left(\int_0^\infty \frac{1}{\sqrt{2 \pi \varepsilon}} \frac{y}{\varepsilon} e^{-y^2/2\varepsilon} \nu_s(dy) \right)^2dx \\
    &+2 \int_0^\infty \psi^2(x) e^{-x^2/ \varepsilon} \frac{x^2}{\varepsilon^2} \left(\int_0^\infty \frac{1}{\sqrt{2 \pi \varepsilon}} e^{-y^2/2\varepsilon} \nu_s(dy) \right)^2dx.
\end{align*}
For the first term on the right-hand side we have
\begin{align*}
    &\int_0^\infty \psi^2(x) e^{-x^2/ \varepsilon} \frac{1}{2 \pi \varepsilon^3}  \left(\int_0^\infty  y e^{-y^2/2\varepsilon} \nu_s(dy) \right)^2dx \\
    &\leq \int_0^\infty \psi^2(x) e^{-x^2/\varepsilon} \frac{1}{2 \pi \varepsilon^2}  dx = O(\varepsilon^{-3/2}e^{-w^2/ 2 \varepsilon})
\end{align*}
where $w = \inf W$. Then we obtain
 \begin{align*}
     \mathbb{E} \left[ \int_0^t \int_0^\infty \psi^2(x) e^{-x^2/ \varepsilon}  \left(\int_0^\infty \frac{1}{\sqrt{2 \pi \varepsilon}} \frac{y}{\varepsilon} e^{-y^2/2\varepsilon} \nu_s(dy) \right)^2dx ds \right]
     = O(\varepsilon^{-3/2} e^{-w^2/2\varepsilon}).
 \end{align*}
We can argue similarly for the second term and hence obtain the desired convergence.
\end{proof}
A similar estimate holds for the elastic correction term $g_\varepsilon^{E,\kappa}$.
\begin{lemma} \label{BoundaryDerivativeLemma2}
Let $(\nu,W^0)$ be the unique solution to the SPDE \eqref{MeasureSPDE} with $\nu$ in the class $\Lambda$. Then we have the convergence
\begin{equation*}
    \mathbb{E} \left[ \int_0^t \int_0^\infty \psi^2(x) \left(\int_0^\infty \partial_x g_\varepsilon(x,y)\nu_s(dy) \right)^2dxds \right] \to 0, \quad \text{as $\varepsilon \to 0$}
\end{equation*}
with the cut-off function $\psi$ specified above.
\end{lemma}

\begin{proof}
The derivative of the elastic correction term $g_\varepsilon^{E,\kappa}$ is given by
\begin{align*}
    \partial_x g_\varepsilon(x,y) = &\partial_x \left(\kappa e^{\kappa(x+y)}e^{\frac{\kappa^2 \varepsilon}{2}} \left(1-\mbox{Erf}\left(\frac{x+y+\kappa \varepsilon}{\sqrt{2 \varepsilon}} \right) \right) \right) \\
    = &2 \kappa^2 e^{\kappa(x+y)} e^{\frac{\kappa^2 \varepsilon}{2}} \frac{1}{\sqrt{2 \pi \varepsilon}} \int_0^\infty \exp\left(-\frac{(z+x+y+\kappa \varepsilon)^2}{2 \varepsilon} \right)dz \\
    &- 2\kappa e^{\kappa(x+y)} e^{\frac{\kappa^2 \varepsilon}{2}} \frac{1}{\sqrt{2 \pi \varepsilon}} \exp \left(-\frac{(y+x+\kappa \varepsilon)^2}{2 \varepsilon} \right).
\end{align*}
We can then estimate
\begin{align*}
    2 \kappa^2 e^{\kappa(x+y)} e^{\frac{\kappa^2 \varepsilon}{2}} \frac{1}{\sqrt{2 \pi \varepsilon}} \int_0^\infty \exp\left(-\frac{(z+x+y+\kappa \varepsilon)^2}{2 \varepsilon} \right)dz \leq \kappa^2 e^{-\frac{(x+y)^2}{2 \varepsilon}} 
\end{align*}
by Lemma \ref{ElasticCorrectionTermEstimate}. For the other term we get
\begin{equation*}
    2\kappa e^{\kappa(x+y)} e^{\frac{\kappa^2 \varepsilon}{2}} \frac{1}{\sqrt{2 \pi \varepsilon}} \exp \left(-\frac{(y+x+\kappa \varepsilon)^2}{2 \varepsilon} \right) \leq \frac{2 \kappa}{\sqrt{2 \pi \varepsilon}} e^{-\frac{(x+y)^2}{2 \varepsilon}}.
\end{equation*}
Thus, we get
\begin{align*}
    &\mathbb{E} \int_0^t \int_0^\infty \psi^2(x) \left(\int_0^\infty \partial_x g_\varepsilon(x,y) \nu_s(dy) \right)^2 dx ds\\
    &\leq  \mathbb{E} \int_0^t \int_0^\infty \psi^2(x) \kappa^4 \left(\int_0^\infty e^{-\frac{(x+y)^2}{2\varepsilon}} \nu_s(dy)  \right)^2dxds \\
    & \quad + \mathbb{E} \int_0^t \int_0^\infty \psi^2(x) 4 \kappa^2 \left(\int_0^\infty \frac{1}{\sqrt{2 \pi \varepsilon}} e^{-\frac{(x+y)^2}{2\varepsilon}} \nu_s(dy) \right)^2dxds \\
    &\leq \mathbb{E} \int_0^t \int_0^\infty \psi^2(x) \kappa^4 e^{-\frac{x^2}{\varepsilon}}dxds + \mathbb{E} \int_0^t \int_0^\infty \psi^2(x) \frac{e^{-x^2/\varepsilon}}{2 \pi \varepsilon} dxds  =  O(\varepsilon^{-1/2}e^{-w^2/2\varepsilon})
\end{align*} 
with $w = \inf W$.
\end{proof}
The necessary controls on the error term $\Bar{\mathcal{E}}^h_{t,\varepsilon}$ are proved in the following lemmas.
\begin{lemma} \label{EErrorEstimate}
Let $\bar{\mathcal{E}}^h_{s,\varepsilon}$ be the error term defined in \eqref{DerivativeError1} and $\psi$ the cut-off function specified above. Then we have the convergence
\begin{equation*}
    \mathbb{E} \left[ \int_0^t \norm{\psi \bar{\mathcal{E}}^h_{s,\varepsilon}}_2^2ds \right] \to 0 \quad \text{as $\varepsilon \to 0$}.
\end{equation*}
\end{lemma}
\begin{proof}
We can combine the estimates in \cite[Lemma 8.2]{HamblyLedger2017}  with the definition of the elastic heat kernel to get
\begin{align*}
    \lvert \bar{\mathcal{E}}_{s,\varepsilon}^h \rvert &\leq C \norm{\partial_{xx}h}_\infty \int_0^\infty \lvert x-y \rvert^2 \lvert \partial_x G_\varepsilon^{E,\kappa} \rvert \nu_s(dy) \\
    &\leq C \norm{\partial_{xx}h}_\infty \int_0^\infty \lvert x-y \rvert^3 \varepsilon^{-1}p_\varepsilon(x-y) \nu_s(dy) \\
    & \quad +C \norm{\partial_{xx}h}_\infty \int_0^\infty \lvert x-y \rvert^2 \lvert \partial_x p_\varepsilon(x+y) \rvert \nu_s(dy) \\
    & \quad + C\norm{\partial_{xx}h}_\infty \int_0^\infty \lvert x-y \rvert^2 \lvert \partial_x g_\varepsilon^{E,\kappa}(x,y) \rvert \nu_s(dy).
\end{align*}
The result follows for the first term by the results from the absorbing case in \cite[Lemma 8.2]{HamblyLedger2017} . Similar to the calculations in Lemma \ref{BoundaryDerivateLemma1} and Lemma \ref{BoundaryDerivativeLemma2} we can show that
\begin{equation*}
    \mathbb{E} \left[ \int_0^t \int_0^\infty \psi^2(x) \left( \int_0^\infty \lvert x-y \rvert^2 \lvert \partial_x g_\varepsilon(x,y)^{E,\kappa} \rvert \nu_s(dy) \right)^2 dx ds \right] = O(\varepsilon^{-1/2}e^{-w^2/2\varepsilon})
\end{equation*}
and
\begin{equation*}
     \mathbb{E} \left[ \int_0^t \int_0^\infty \psi^2(x) \left( \int_0^\infty \lvert x-y \rvert^2 \lvert \partial_x p_\varepsilon(x+y) \rvert \nu_s(dy) \right)^2 dx ds \right] = O(\varepsilon^{-5/2}e^{-w^2/2\varepsilon}).
\end{equation*}
\end{proof}
To simplify the notation, we denote by $o^\psi_{sq}(1)$ any family of functions $\{(f_{t,\varepsilon})_{t \in [0,T]} \}_{\varepsilon>0}$ such that
\begin{equation*}
    \mathbb{E} \left[ \int_0^T \norm{\psi f_{t,\varepsilon}}^2_2 dt \right] \to 0, \qquad \text{as $\varepsilon \to 0$}.
\end{equation*}
The final auxiliary result is a control on the other error term $\mathcal{H}_{t,\varepsilon}$.
\begin{lemma} \label{HErrorEstimate}
Let $\mathcal{H}_{t,\varepsilon}$ be the error term defined in \eqref{DerivativeError2}. Then there is a constant $c_{\mathcal H}$ such that
\begin{equation*}
    \lvert \mathcal{H}_{t,\varepsilon} \rvert \leq c_{\mathcal H} \lvert T_{2\varepsilon}^{E,\kappa} \nu_t \rvert + o_{sq}^\psi(1)
\end{equation*}
for $\varepsilon>0$ small enough.
\end{lemma} 
\begin{proof}
We can use the definition of the elastic heat kernel $G_\varepsilon^{E,\kappa}$ to estimate
\begin{equation*}
    \lvert \mathcal{H}_{t,\varepsilon} \rvert \leq \lvert \nu_t((x-y)\partial_x p_\varepsilon(x-y)) \rvert + \lvert \nu_t((x-y)\partial_x p_\varepsilon(x+y)) \rvert + \lvert \nu_t((x-y)\partial_x g_\varepsilon^{E,\kappa}(x,y)) \rvert.
\end{equation*}
For the first term there is a constant $c$ such that
\begin{equation*}
    \lvert \nu_t((x-y)\partial_x p_\varepsilon(x-y)) \rvert \leq c \nu_t(p_{2\varepsilon}(x-y))
\end{equation*}
by \cite[Lemma 8.2]{HamblyLedger2017}. Moreover, we have the estimate
\begin{equation*}
    g_\varepsilon^{E,\kappa}(x,y) \leq \kappa e^{-\frac{(x+y)^2}{2\varepsilon}} \leq p_{\varepsilon}(x+y),
\end{equation*}
for $\varepsilon$ small enough. Thus there is a constant $C$ such that
\begin{equation*}
    \lvert \nu_t((x-y)\partial_x p_\varepsilon(x-y)) \rvert \leq C  T_{2\varepsilon}^{E,\kappa} \nu_t. 
\end{equation*}
The estimates for the other two terms follows as in Lemma \ref{BoundaryDerivateLemma1} and Lemma \ref{BoundaryDerivativeLemma2}.
\end{proof}
Now we have all the tools to prove the existence of a density in the interior.
\begin{proposition} \label{DensityInterior}
Let $(\nu,W^0)$ be the unique solution to the SPDE \eqref{MeasureSPDE} with $\nu$ in the class $\Lambda$.  The measure $\nu_t$ restricted to $(0,\infty)$ has a density $V_t$ for all $t \in [0,T]$ with probability 1. The density is square integrable on $(\delta,\infty)$ for every $\delta>0$.
\end{proposition}

\begin{proof}
We proceed similarly to the uniqueness proof in Section \ref{Sec:LinUniq}. We use the test function $y \mapsto G_\varepsilon^{E,\kappa}(x,y)$ in the SPDE \eqref{MeasureSPDE}, switch derivatives from $y$ to $x$ and take the coefficients outside the integration against the measure. This creates the error terms defined in \eqref{DerivativeError1} and \eqref{DerivativeError2}. Applying It{\^o}'s formula we can then derive the dynamics

\begin{align*}
    d(T_\varepsilon \nu_t)^2 =& - 2 T_\varepsilon^{E,\kappa} \nu_t \left(\mu_t \partial_x T_\varepsilon^{E,\kappa} \nu_t - \partial_x \mu_t \mathcal{H}_{t,\varepsilon} + \bar{\mathcal{E}}_{t,\varepsilon}^\mu \right)dt \\
    &+ T_\varepsilon^{E,\kappa} \nu_t \partial_x \left(\sigma_t^2 \partial_x T_\varepsilon^{E,\kappa} \nu_t - \partial_x \sigma_t^2 \mathcal{H}_{t,\varepsilon} + \bar{\mathcal{E}}_{t,\varepsilon}^{\sigma^2} \right) dt \\
    &- 2 T_\varepsilon^{E,\kappa} \nu_t \rho_t \left(\sigma_t \partial_x T_\varepsilon^{E,\kappa} \nu_t - \partial_x \sigma_t \mathcal{H}_{t,\varepsilon} + \bar{\mathcal{E}}_{t,\varepsilon}^\sigma \right)dW_t^0 \\
    &+4 T_\varepsilon^{E,\kappa} \partial_x \nu_t(\mu_t p_\varepsilon(x+\cdot))dt + 4 T_\varepsilon^{E,\kappa} \nu_t \partial_x \nu_t(\rho_t \sigma_t p_\varepsilon(x+\cdot))dW_t^0 \\
    &-4 T_\varepsilon^{E,\kappa} \nu_t \partial_x \nu_t(\mu_t g_\varepsilon^{E,\kappa}(x,\cdot))dt - 4 T_\varepsilon^{E,\kappa} \nu_t \partial_x \nu_t(\rho_t \sigma_t g_\varepsilon^{E,\kappa}(x,\cdot))dW_t^0 \\
    &+ \left(\rho_t \left(\sigma_t \partial_x T_\varepsilon^{E,\kappa} \nu_t - \partial_x \sigma_t \mathcal{H}_{t,\varepsilon} + \bar{\mathcal{E}}_{t,\varepsilon}^\sigma \right)+2 \partial_x \nu_t(\rho_t \sigma_t p_\varepsilon(x+\cdot))- 2 \partial_x \nu_t(\rho_t \sigma_t g_\varepsilon(x,\cdot)) \right)^2dt.
\end{align*}

Next, we multiply the equation by $\psi^2$ and integrate over $x$. We then estimate all the terms on the right-hand side individually using the results of Lemma \ref{BoundaryDerivateLemma1}, Lemma \ref{BoundaryDerivativeLemma2}, Lemma \ref{EErrorEstimate} and Lemma \ref{HErrorEstimate} together with the generalized Young's inequality. For the first term we have

\begin{align*}
    &-2 \int_0^t \int_0^\infty  \psi^2 T_\varepsilon^{E,\kappa} \nu_s \left( \mu_s \partial_x T_\varepsilon^{E,\kappa} \nu_s- \partial_x \mu_s \mathcal{H}_{s,\varepsilon} + \Bar{\mathcal{E}}_{s,\varepsilon}^\mu \right)dxds \\
    &\leq c_\eta \int_0^t \norm{\psi T_\varepsilon^{E,\kappa} \nu_s}_2^2ds + c_\eta  \int_0^t \norm{\psi T_{2 \varepsilon}^{E,\kappa} \nu_s}_2^2ds + \eta \int_0^t \int_0^\infty \mu_s^2 \psi^2 \lvert \partial_x T_\varepsilon^{E,\kappa} \nu_s \rvert^2dxds +o(1).
\end{align*}

For the term in the second line we apply integration by parts and obtain
\begin{align*}
    & \int_0^t \int_0^\infty \psi^2 T_\varepsilon^{E,\kappa} \nu_s \partial_x \left(\sigma^2 \partial_x T_\varepsilon^{E,\kappa} \nu_s - \partial_x \sigma^2_s \mathcal{H}_{s,\varepsilon}+\bar{\mathcal{E}}_{s,\varepsilon}^{\sigma^2}  \right)dxds \\
    &= - \int_0^t \int_0^\infty \psi^2 \sigma^2 \lvert \partial_x T_\varepsilon^{E,\kappa} \nu_s \rvert^2 dxds - 2\int_0^t \int_0^\infty \psi \psi^\prime T_\varepsilon^{E,\kappa} \nu_s \partial_x T_\varepsilon^{E,\kappa} \nu_s \sigma^2_s dxds \\
    &\quad+\int_0^t \int_0^\infty \psi^2 \partial_x T_\varepsilon^{E,\kappa} \nu_s \partial_x \sigma_s^2 \mathcal{H}_{s,\varepsilon} dxds + 2\int_0^t \int_0^\infty \psi \psi^\prime T_\varepsilon^{E,\kappa} \nu_s \partial_x \sigma_s^2 \mathcal{H}_{s,\varepsilon} dxds \\
    &\quad-  \int_0^t \int_0^\infty \psi^2 \partial_x T_\varepsilon^{E,\kappa} \nu_s \bar{\mathcal{E}}_{s,\varepsilon}^{\sigma^2} - 2 \int_0^t \int_0^\infty \psi \psi^\prime T_\varepsilon^{E,\kappa} \nu_s \bar{\mathcal{E}}_{s,\varepsilon}^{\sigma^2}dxds \\
    &\leq -  \int_0^t \int_0^\infty \psi^2 \sigma^2 \lvert \partial_x T_\varepsilon^{E,\kappa} \nu_s \rvert^2 dxds + \eta \int_0^t \int_0^\infty \psi^2 \lvert \partial_x T_\varepsilon^{E,\kappa} \nu_s \rvert^2 dxds  \\
    &\quad + c_\eta \int_0^t \int_0^\infty \lvert T_\varepsilon^{E,\kappa} \nu_s \rvert^2dxds + c_\eta  \int_0^t \int_0^\infty \lvert T_{2\varepsilon}^{E,\kappa} \nu_s \rvert^2 dxds +o (1).
\end{align*}

For now, we leave the terms involving stochastic integrals unchanged. The remaining terms in the fourth and fifth line can be estimated via
\begin{align*}
    &4 \int_0^t \int_0^\infty \psi^2 T_\varepsilon^{E,\kappa} \nu_s \partial_x \nu_s(\mu_s p_\varepsilon(x+\cdot))dx ds + 4 \int_0^t \int_0^\infty \psi^2 T_\varepsilon^{E,\kappa} \nu_s \partial_x \nu_s(\mu_s g_\varepsilon^{E,\kappa}(x,\cdot))dx ds\\
    &\leq c_\eta \int_0^t \norm{\psi T_\varepsilon^{E,\kappa} \nu_s}_2^2ds  + o(1).
\end{align*}
We can write the last term as
\begin{align*}
    & \int_0^t \int_0^\infty \psi^2 \Big(\rho_s \left(\sigma_s \partial_x T_\varepsilon^{E,\kappa} \nu_s - \partial_x \sigma_s \mathcal{H}_{s,\varepsilon} + \bar{\mathcal{E}}_{s,\varepsilon}^\sigma \right)+ o_{sq}^\psi(1) \Big)^2 dxds \\
    &= \int_0^t \int_0^\infty \psi^2\rho_s^2 \left(\sigma_s \partial_x T_\varepsilon^{E,\kappa} \nu_s - \partial_x \sigma_s \mathcal{H}_{s,\varepsilon}
    + \bar{\mathcal{E}}_{s,\varepsilon}^\sigma \right)^2 dxds \\
    &\quad+  \int_0^t \int_0^\infty \psi^2 \rho_s \left(\sigma_s \partial_x T_\varepsilon^{E,\kappa} \nu_s - \partial_x \sigma_s \mathcal{H}_{s,\varepsilon} + \bar{\mathcal{E}}_{s,\varepsilon}^\sigma \right) o_{sq}^\psi(1)dxds + o(1) \\
    &\leq (1+\eta)  \int_0^t \int_0^\infty \psi^2 \rho_s^2 \sigma_s^2 \lvert \partial_x T_\varepsilon^{E,\kappa} \nu_s \rvert^2 dxds\\
    &\quad-(1+\eta)  \int_0^t \int_0^\infty \psi^2 \rho_s^2 \sigma_s \partial_x T_\varepsilon^{E,\kappa} \nu_s \left(\partial_x \sigma_s \mathcal{H}_{s,\varepsilon}-\bar{\mathcal{E}}^\sigma_{s,\varepsilon} \right)dxds \\
    &\quad+ (1+\eta)  \int_0^t \int_0^\infty \psi^2 \rho^2 \left( \partial_x \sigma \mathcal{H}_{s,\varepsilon}-\bar{\mathcal{E}}^\sigma_{s,\varepsilon} \right)^2dxds + o(1).
\end{align*}
We obtain 
\begin{align*}
    &\int_0^t \int_0^\infty \psi^2 \Big(\rho_s \left(\sigma_s \partial_x T_\varepsilon^{E,\kappa} \nu_s - \partial_x \sigma_s \mathcal{H}_{s,\varepsilon} + \bar{\mathcal{E}}_{s,\varepsilon}^\sigma \right)+ o_{sq}^\psi(1) \Big)^2 dxds\\
    &\leq (1+2\eta+\eta^2) \int_0^t \int_0^\infty \psi^2 \rho_s^2 \sigma_s^2 \lvert \partial_x T_\varepsilon^{E,\kappa} \nu_s \rvert^2 dxds + c_\eta \int_0^t \norm{\psi T_{2\varepsilon}^{E,\kappa}\nu_s}_2^2ds + o(1).    
\end{align*}
Putting everything back together and choosing $\eta>0$ small enough such that
\begin{equation*}
    \sigma_s^2 - \rho_s^2(1+2\eta+\eta^2) \sigma^2_s-\eta \mu_s^2 -\eta \geq c_0
\end{equation*}
for some $c_0>0$ yields  
\begin{align*}
    \norm{ \psi T_\varepsilon^{E,\kappa} \nu_t}_2^2 
    \leq& \norm{T_\varepsilon^{E,\kappa}\nu_0}_2^2 + c_\eta \int_0^t \norm{T_\varepsilon^{E,\kappa}\nu_s}_2^2ds + c_\eta \int_0^t \norm{T_{2\varepsilon}^{E,\kappa}\nu_s}_2^2ds + o(1)\\
    &- 2 \int_0^t \int_0^\infty \psi^2 T_\varepsilon^{E,\kappa} \nu_s \rho_s \left(\sigma_s \partial_x T_\varepsilon^{E,\kappa} \nu_s - \partial_x \sigma_t \mathcal{H}_{t,\varepsilon} + \bar{\mathcal{E}}_{t,\varepsilon}^\sigma \right)dxdW_s^0 \\
    &+ 4 \int_0^t \int_0^\infty \psi^2 T_\varepsilon^{E,\kappa} \nu_s \partial_x \nu_s(\rho_s \sigma_s p_\varepsilon(x+\cdot))dxdW_s^0\\
    &- 4 \int_0^t \int_0^\infty \psi^2 T_\varepsilon^{E,\kappa} \nu_s \partial_x \nu_s(\rho_s \sigma_s g_\varepsilon^{E,\kappa}(x,\cdot))dxdW_s^0.
\end{align*}
Next, we take supremum over $t$ and then expectation. Using \cite[Lemma 8.5]{HamblyLedger2017}  to estimate the stochastic integrals, which can easily be adapted to this case, we get 
\begin{align*}
    \mathbb{E} \left[ \sup_{s \in [0,t]} \norm{\psi T_\varepsilon^{E,\kappa}\nu_s }_2^2 \right] \leq& c^\prime \mathbb{E} \left[ \norm{T_\varepsilon^{E,\kappa}\nu_0}_2^2 \right] + c^\prime \mathbb{E} \left[ \int_0^t \norm{T_\varepsilon^{E,\kappa}\nu_s}_2^2ds \right]\\
    &+ c^\prime \mathbb{E} \left[ \int_0^t \norm{T_{2\varepsilon}^{E,\kappa}\nu_s}_2^2ds \right]+ o(1)
\end{align*}
with a constant $c^\prime>0$. Taking the limit $\varepsilon \to 0$ and using Proposition \ref{IntegratedSmoothedL2}, Assumption \ref{AssumptionInitial} and Lemma \ref{L2contraction} we obtain 
\begin{equation}
    \lim_{\varepsilon \to 0} \mathbb{E}  \left[\sup_{s \in [0,t]}\norm{T_\varepsilon^{E,\kappa} \nu_s}_{L^2(U)}^2 \right] \leq \lim_{\varepsilon \to 0} \mathbb{E}  \left[\sup_{s \in [0,t]} \norm{\psi T_\varepsilon^{E,\kappa} \nu_s}_2^2 \right] < \infty.
\end{equation}
Thus, we have the existence of an $L^2$-density in the set $U$ using the same arguments as in the proof of Proposition \ref{PropIntegratedDensityRegularity} and since $U$ was arbitrary the result follows.
\end{proof}

\subsection{Ruling out a Dirac mass at the origin}

As a special case of Theorem \ref{probrep} in the next section, we can characterize the unique solution $(\nu, W^0)$ to the SPDE \eqref{MeasureSPDE} as the conditional law
\begin{equation}\label{eq:diffusion_char}
\nu_t(dx) = \mathbb{P} \left( X_t \in  dx,\, t < \tau \right \vert \mathcal{F}_t^{W^0})
\end{equation}
of a given reflected diffusion
\begin{equation}\label{eq:reflected_SDE}
dX_t = \mu(t,X_t)dt + \sigma(t,X_t) \rho(t) W_t^0 + \sigma(t,X_t) (1-\rho(t)^2)^{\frac{1}{2}}dW_t+dL_t,
\end{equation}
where $W$ is a Brownian motion independent of $W^0$ and $\tau$ is the elastic killing time defined by a standard exponential random variable independent of $X$. This point of view makes it easy to deduce from \cite{Burdzy2003} that there cannot be an atom at the origin.

\begin{proof}[Proof of Proposition \ref{prop:Burdzy}]
Express the unique solution $\nu$ as \eqref{eq:diffusion_char}. We can then apply the scale transformation $\zeta$ defined in Lemma \ref{Scale} to \eqref{eq:reflected_SDE} in order to obtain
\begin{equation*}
    \nu_t(A) \leq \mathbb{P} \left(Z_0 + \rho W_t^0 + \sqrt{1-\rho^2} W_t^1 + \int_0^t \tilde{\mu}_s ds + L_t \in \zeta(t,A) \big\vert \mathcal{F}_t^{W^0}  \right),
\end{equation*}
for any Borel set $A$ in $[0,\infty)$. By Lemma \ref{Scale}, we have that 
\begin{equation} \label{relfectionIneq1}
    Z_0 + \rho W_t^0 + \sqrt{1-\rho^2}W_t^1 + \int_0^t \tilde{\mu}_s ds \geq Z_0 + \rho W_t^0 + \sqrt{1-\rho^2}W_t^1 - Ct  
\end{equation}
for all $t \in [0,T]$ almost surely with a constant $C>0$. Note that the processes in \eqref{relfectionIneq1} only differ by the drift term. Moreover the drift $-Ct$ is always decreasing by more than the drift $\int_0^t \tilde{\mu}_sds$.
Thus we have
\begin{align*}
    &\left\{Z_0 + \rho W_t^0 + \sqrt{1-\rho^2}W_t^1 + \int_0^t \tilde{\mu}_s ds +L_t =0  \right\}\\
    &= \left\{Z_0 + \rho W_t^0 + \sqrt{1-\rho^2}W_t^1 + \int_0^t \tilde{\mu}_s ds = \inf_{u\leq t} \left\{Z_0 + \rho W_u^0 + \sqrt{1-\rho^2}W_u^1 + \int_0^u \tilde{\mu}_s ds  \right\} \right\} \\
    &\subseteq \left\{Z_0 + \rho W_t^0 + \sqrt{1-\rho^2}W_t^1 -Ct = \inf_{u\leq t} \left\{Z_0 + \rho W_u^0 + \sqrt{1-\rho^2}W_u^1 -uC  \right\} \right\} \\
    &=\left\{Z_0 + \rho W_t^0 + \sqrt{1-\rho^2}W_t^1 -Ct +\tilde{L}_t =0  \right\},
\end{align*}
where $\tilde{L}$ is the local time of the process in the last line.
As a result we obtain

\begin{equation*}
    0 \leq \nu_t(\{0\}) \leq \mathbb{P} \left(Z_0 + \rho W_t^0 + \sqrt{1-\rho^2}W_t^1 - Ct + \tilde{L}_t =0  \Big\vert \mathcal{F}_t^{W^0} \right),
\end{equation*}
for all $t \in [0,T]$ almost surely. With $f(t) :=Ct$, we obviously have that
$f(t)/\sqrt{t}$ is non-decreasing, so $f$ is not in the upper class of Brownian motion by \cite[p.144]{Knight1981}. Thus, \cite[Theorem 2.2 and Theorem 2.5]{Burdzy2003} gives that the probability on the right-hand side is zero for all $t \in [0,T]$ almost surely, and hence $\nu_t$ cannot have an atom at zero for any $t$. Combining this with Proposition \ref{DensityInterior} we can conclude that we have a density everywhere on the positive half-line.
\end{proof}
While the previous theorem establishes existence of a density for the entire positive half-line, we do not get the square integrability all the way up to the boundary, beyond a set of times of full measure as guaranteed by Theorem \ref{thm:L2-density}.
It is unclear to us how to extend this result to include all $t \in [0,T]$ and we actually conjecture it to be false for all $t \in [0,T]$. To see where the problem lies we consider the simpler case of a reflecting boundary and set the drift $\mu$ to zero and the volatility $\sigma$ to one. The solution to the SPDE then takes the form
\begin{equation*}
    \nu_t(dx) = \mathbb{P} \left(\rho W_t^0 + \sqrt{1-\rho^2} W_t^1+L_t \in dx\, \vert\, \mathcal{F}^{W^0}_t \right).
\end{equation*}
We can interpret this conditional distribution as the distribution of the scaled Brownian motion $\sqrt{1-\rho^2} W^1$ reflected on a given fixed Brownian path $-\rho W^0$. This is a particular instance of a Brownian motion reflected in a time-dependent domain, which, together with the related heat equation, have been studied in a series of papers \cite{ Burdzy2004b,Burdzy2003, Burdzy2004a, Burdzy2002}. In \cite{Burdzy2003} the authors show that, at the boundary of such domains, both singularities of the density as well as atoms are possible. While they establish that there are no atoms when the boundary is a path of a Brownian motion, in \cite{Burdzy2002} it is shown that in this case the density exhibits blow-up at the boundary on a dense subset of times $t$. These effects arise when the boundary sharply moves into the domain and the diffusion cannot transport the heat away fast enough. The authors show that, for any boundary that locally moves in sharper than Brownian motion (i.e.,~is in the upper class of Brownian motion), an atom occurs. This leads us to suspect that the blow-up is too strong to allow for square integrability of the density at the origin.  

\section{Extension to nonlinear interactions in the drift}
In this section we discuss some of the results we can obtain and some of the difficulties associated to a natural extension of the particle system which allows for interaction through a dependence on the empirical measure $\nu^N_t$ in the drift coefficient $\mu$. The particle system then becomes
\begin{align} \label{particle2 SDE}
\begin{split}
    X_t^{i,N} &= X_0^i + \int_0^t \mu(s,X_s^{i,N},\nu_s^N)ds + \int_0^t \sigma(s,X_s^{i,N}) \rho(s) dW_s^0 \\
    &+ \int_0^t \sigma(s,X_s^{i,N})(1- \rho(s)^2)^{\frac{1}{2}}dW_s^i+L_t^{i,N},
\end{split}
\end{align}
together with the elastic stopping times $\tau^{i,N}$. In this case we need additional assumptions on the regularity of $\mu$ with respect to the measure variable.

\begin{assumption}[Regularity in the Measure Variable] \label{AssumptionMeanField}
The drift coefficient $\mu$ is Lipschitz continuous in the measure variable $\nu$ on the space of sub-probability measures $\mathbf{M}_{\leq 1}(\R)$ with respect to the bounded Lipschitz distance $d_0$ which is given by
\begin{equation}
    d_0(\nu, \Tilde{\nu}) \coloneqq \sup \{\lvert \langle \psi, \nu - \Tilde{\nu} \rangle \rvert : \norm{\psi}_{Lip} \leq 1, \norm{\psi}_{\infty}  \leq 1  \},
\end{equation}
i.e., there exists a fixed constant $c>0$ such that $\lvert \mu(t,x,\nu)- \mu(t,x,\Tilde{\nu}) \rvert \leq c d_0 (\nu, \tilde{\nu})$.
\end{assumption}

Through the empirical measure in the drift coefficient the elastic stopping times $\tau^{i,N}$ become part of the equations for the particles. Thus, well-posedness of the particle system needs additional arguments compared to the system  \eqref{particle SDE}. The existence of such a particle system can be established by adjusting the arguments in  \cite[Theorem 2.3]{SojmarkWorkingPaper22} to this case. We need to add a common noise in the filtrations defined therein. Another difference is that in our case the interactions are through the empirical measure not the loss function. An important step in the construction in \cite[Section 3]{SojmarkWorkingPaper22} is, for a given particle $i$, to consider the system without particle $i$ given that $t < \tau^{i,N}$. In our case such a system will, in contrast to the system in \cite{SojmarkWorkingPaper22}, not be independent of the particle $i$ because of the empirical-measure interaction. However, given the condition $ t <\tau^{i,N}$ we can consider this system with the empirical measure
\begin{equation*}
    \nu_t^{N,-i}(dx) = \frac{1}{N} \left(\sum_{j=1, j\neq i}^N \delta_{X_t^{j,N}}(dx)\mathbbm{1}_{t< \tau^{j}} + \delta_{X_t^{i,N}}(dx) \right).
\end{equation*}
Thus, the system without particle $i$ depends on $W^i$ and $X^i_0$ but not $\chi^i$. Considering this fact, the construction in \cite{SojmarkWorkingPaper22} carries through in our case. Using Assumptions \ref{AssumptionInitial},\ref{AssumptionSpace} and \ref{AssumptionMeanField} a standard argument relying on the Lipschitz continuity and Gronwall's lemma can be used to establish uniqueness.

Following the same arguments as in Section \ref{existenceSection} we obtain the existence result in this case. The methods we used in the proofs in Section \ref{existenceSection} can easily be generalised to allow for interaction in the drift. In particular the martingale approach can be extended in a straight forward way to $\mu$ depending also on $\xi$ in \eqref{martingaleComponents}. 
\begin{theorem} \label{NonLinExistenceThm}
Let $(\nu^N,W^0)$ be a sequence with $\nu^N$ being empirical-measure processes corresponding to the particle system (\ref{particle2 SDE}) that satisfy Assumptions \ref{AssumptionInitial}, \ref{AssumptionSpace} and \ref{AssumptionMeanField}. Then $(\nu^N,W^0)$ possesses converging subsequences in $(D_{\mathcal{S}^{\prime}},M1) \times (C_{\R},\norm{}_{\infty})$. Moreover, for any limit point $(\nu,W^0)$ the process $\nu$ is in the class $\Lambda$ and $(\nu,W^0)$ satisfies the SPDE
\begin{align} \label{MeasureSPDE2}
    \begin{split}
        \langle \nu_t,\phi \rangle = \langle \nu_0, \phi \rangle &+ \int_0^t  \langle \nu_s,\mu(s,\cdot,\nu_s)  \phi^\prime \rangle ds + \frac{1}{2} \int_0^t \langle \nu_s, \sigma^2(s,\cdot)  \phi^{\prime \prime} \rangle ds\\
        &+ \int_0^t \langle \nu_s, \rho(s) \sigma(s,\cdot)  \phi^\prime \rangle dW_s^0
    \end{split}
\end{align}
for all times $t \in [0,T]$ and all test functions $\phi \in \mathcal{C}^{E,\kappa}_0(\mathbb{R})$, where
\[
   \mathcal{C}^{E,\kappa}_0(\mathbb{R}) = \{\phi \in \mathcal{S}: \partial_x \phi(0) = \kappa \phi(0)  \}.
\]
\end{theorem}
Note that in this case the limit equation becomes a nonlinear SPDE with a nonlinearity in the drift term $\mu$. A convenient way of obtaining regularity for solutions to this nonlinear SPDE is to use the following probabilistic representation.

\begin{theorem} \label{probrep}
Let $(\nu,W^0)$ be a solution to the SPDE \eqref{MeasureSPDE2} with $\nu$ in the class $\Lambda$.  Then, for all $t \in [0,T]$ we have the following representation of $\nu_t$
\begin{equation}
    \nu_t = \mathbb{P} \left( X_t \in  \cdot, t < \tau \right \vert \mathcal{F}_t^{\nu,W^0}),
\end{equation}
where $X$ is a reflecting particle with dynamics given by
\begin{equation*}
    dX_t = \mu(t,X_t,\nu_t)dt + \sigma(t,X_t) \rho_t W_t^0 + \sigma(t,X_t) (1-\rho_t^2)^{\frac{1}{2}}dW_t+dL_t,
\end{equation*}
with $W^0,W$ independent Brownian motions and $X_0$ being distributed according to $\nu_0$. The stopping time  $\tau$ is the elastic stopping time associated with $X$ and $(\mathcal{F}^{\nu,W^0}_t)$ the filtration generated by $\nu$ and $W^0$.
\end{theorem}

\begin{proof} 
The proof is based on a decoupling argument and the linear uniqueness result. Let $\nu$ be a weak solution to the nonlinear SPDE  which exists by Theorem \ref{NonLinExistenceThm}.  We then define a stochastic process $X$ as the solution to the SDE
\begin{equation}
    dX_t = \mu_t^{\nu}dt +  \sigma(t,X_t) (\rho_t dW_t^0 + (1-\rho_t^2)^{1/2}dW_t^1)+ dL_t, 
\end{equation}
where the drift term $\mu^{\nu}$ is given by
\begin{equation*}
    \mu_t^{\nu} \coloneqq \mu(t,X_t,\nu_t)
\end{equation*}
and $W^0,W^1$ are independent Brownian motions. The associated elastic stopping time is $\tau^{lin} \coloneqq \inf\{t>0:  L_t > \xi  \}$ for an independent exponential random variable $\xi \sim \mbox{Exp}(\kappa)$. Define a process $(\nu_t^{lin})$ taking values in the space of sub-probability measures as
\begin{equation}
    \langle \nu_t^{lin}, \phi \rangle \coloneqq \mathbb{E} \left[\phi(X_t) \mathbbm{1}_{t < \tau^{lin}}  \vert \mathcal{F}_t^{\nu,W^0} \right]
\end{equation}
for $\phi \in \mathcal{S}$. Take a test function $\phi \in \mathcal{C}^{E,\kappa}_0(\mathbb{R})$, apply It{\^o}'s formula, rearrange and take conditional expectation to find that $\nu_t^{lin}$ solves the linear SPDE
\begin{align} \label{linearSPDE}
\begin{split}
  d \langle \nu_t^{lin}, \phi \rangle = &\langle \nu_t^{lin}, \mu_t^{\nu}  \phi^\prime \rangle dt + \frac{1}{2} \langle \nu_t^{lin}, \sigma^2_t  \phi^{\prime \prime} \rangle dt \\
  &+ \langle \nu_t^{lin}, \rho_t \sigma_t  \phi^\prime \rangle dW_t^0
  \end{split}
\end{align}
for all $\phi \in \mathcal{C}^{E,\kappa}_0(\mathbb{R})$. By the same arguments as for the particle system in Section \ref{probEstSection}, one easily checks that the solution is in the class $\Lambda$. Since $\nu$ is a solution to the nonlinear SPDE, it also solves the SPDE (\ref{linearSPDE}). By Therorem \ref{uniqueness} solutions to (\ref{linearSPDE}) are unique in the class $\Lambda$, so we have $\nu_t =\nu_t^{lin}$. This yields the desired result.
\end{proof}
Based on this representation the same regularity results as in Theorem \ref{thm:L2-density} and Proposition \ref{prop:Burdzy} can be shown using the methods in Section \ref{DensityRegularitySection}. The question of uniqueness of solutions to the nonlinear SPDE is more subtle. In the case of an absorbing boundary, uniform $L^2$ regularity is used to deal with the nonlinearity by employing an argument involving a sequence of stopping times (see \cite{HamblyLedger2017, HamblySojmark2019}). As discussed in Section \ref{DensityRegularitySection} we do not expect such regularity to hold in the elastic case. This makes the extension to the nonlinear case appear significantly more difficult than in the absorbing case and an entirely different approach may be needed.

\section{Absorption and Reflection as Limiting Cases}

Intuitively, an elastic boundary condition acts as a mixture of an absorbing boundary and a reflecting boundary -- with the positive parameter $\kappa$ controlling the balance of the two. This can be seen from the elastic condition
\begin{equation}
    \partial_x \phi(0) = \kappa \phi (0),
\end{equation}
if we take the limits $\kappa \to 0$ and $\kappa \to \infty$. When taking $\kappa$ to $\infty$ we obtain the absorbing boundary condition $\phi(0) =0 $ and taking $\kappa$ to $0$ yields the reflecting boundary condition $\partial_x \phi(0) = 0$. This shows that the absorbing and reflecting cases can be obtained as limits of the elastic case. Moreover, we observe the same when analysing the elastic stopping times 
\begin{equation*}
    \tau = \inf \{t>0 : L_t > \chi^{\kappa} \}, \quad \chi^{\kappa} \sim \mbox{Exp}(\kappa).
\end{equation*}
In the limit $\kappa \to \infty$ we get $\chi^{\kappa} \to 0$ and the stopping time $\tau$ becomes the first time the process $X$ hits the boundary at 0, $i.e.$ the absorbing stopping time. For $\kappa \to 0$ we have the convergence $\chi^{\kappa}  \to \infty$ and the elastic killing time $\tau^i$ converges to $\infty$. As a result the limit process becomes purely reflecting. The goal of this section is to show that we also have this convergence at the level of the measure-valued processes as we let $\kappa$ go to $0$ or $\infty$. We will write $\nu^{\kappa}$ for the solution to the elastic SPDE to emphasis the dependence on the elastic-killing parameter $\kappa$. 

SPDEs of the type (\ref{MeasureSPDE}) with an absorbing boundary are studied in \cite{HamblyLedger2017} and in \cite{HamblySojmark2019}. Existence and uniqueness of solutions is proved in these articles. The absorbing nature of the boundary is defined through the space of test functions
\begin{equation}
\mathcal{C}^{A}_0(\mathbb{R}) \coloneqq \{\phi \in \mathcal{S}(\R) : \phi(0)=0 \}.
\end{equation}
By using the space of test functions
\begin{equation}
    \mathcal{C}^{R}_0(\mathbb{R}) \coloneqq \{ \phi \in \mathcal{S}(\R) : \partial_x \phi (0) = 0 \}
\end{equation}
in the formulation of the SPDE \eqref{MeasureSPDE} we have instead a reflecting boundary. With a few adjustments to the proof for the elastic case, we can also prove Theorem \ref{UniquenessReflecting}.

\begin{proof} [Proof of Theorem \ref{UniquenessReflecting}]
The existence and necessary regularity of the weak solution follows through a particle approximation as done in Section \ref{probEstSection} and Section \ref{existenceSection}. Note that the required estimates for the reflecting system were established alongside the estimates for the elastic system in Section \ref{probEstSection}. 

The kernel smoothing method with the reflecting heat kernel $G_\varepsilon^R$ can be employed to prove uniqueness. In the reflecting case the weak formulation of the boundary simplifies because the process takes values in the space of probability measures. Denoting by $T_\varepsilon^R$ the convolution operator with $G_\varepsilon^R$ and by $\nu^0$ a solution to the reflecting SPDE we see that for the anti-derivative at the boundary,  an application of Fubini's theorem and the fact that $G^R_\varepsilon$ integrates to 1 gives $\partial_x^{-1} T_{\varepsilon}^R \nu^0_t(0)=-\nu^0_t[0,\infty)$. The rest of the proof now follows with the same arguments as in the elastic case.
\end{proof}

Now that we have existence and uniqueness for the three boundary cases,
we can prove Theorem \ref{CaseConvergenceThm} on the weak convergence of the solution to the elastic equation to the absorbing and reflecting counter parts in the following way
\begin{itemize}
    \item verify tightness of $(\nu^{\kappa})$ to establish existence of limit points 
    \item show that limit points solve the absorbing or reflecting SPDEs, respectively, 
    \item deduce weak convergence from uniqueness of solutions.
\end{itemize}

The first step concerns tightness. As in the case of the particle approximations we consider tightness on the space of c{\`a}dl{\`a}g processes with values in the space of tempered distributions.
\begin{lemma} \label{tightnessKappa}
Let $(\nu^{\kappa},W^0)$ be a sequence of solutions with $\nu^\kappa$ in the class $\Lambda$ to the SPDE \eqref{MeasureSPDE} with elastic boundary condition for parameter $\kappa>0$. Then, the sequence $(\nu^{\kappa},W^0)$ is tight on $(D_{\mathcal{S}^{\prime}},M1) \times (C_{\R}, \norm{\cdot}_{\infty})$ for both cases, $\kappa \to \infty$ and $\kappa \to 0$ .
\end{lemma}

\begin{proof}
We follow the same approach as in Proposition \ref{TightnessN}. 
By the probabilistic representation from Proposition \ref{probrep} we know that for $\phi \in \mathcal{S}$ we have

\begin{equation}
    \langle \nu_t^{\kappa},\phi \rangle = \mathbb{E} \left[ \phi(X_t) \mathbbm{1}_{t < \tau^{\kappa}} \vert \mathcal{F}_t^{\nu^\kappa,W^0} \right].
\end{equation}
We can then again apply the decomposition of \cite[Proposition 4.2]{Ledger2016} and obtain that the first condition we need to show is
\begin{align} \label{KappaTightnessCondition1}
\begin{split}
    \mathbb{E} \left[ \lvert \langle \hat{\nu}_t^{\kappa},\phi \rangle - \langle \hat{\nu}_s^{\kappa},\phi \rangle \rvert^4 \right]
    =  O(\lvert t-s \rvert^2), \qquad{\text{as } \lvert t-s \rvert \to 0}
\end{split}
\end{align}
for 
\begin{equation*}
    \langle \hat{\nu}_t^\kappa, \phi \rangle = \mathbb{E} \left[\phi(X_{t \wedge \tau^\kappa}) \vert \mathcal{F}_t^{\nu^\kappa,W^0} \right].
\end{equation*}
The definition of $X$ is independent of $\kappa$ and we can conclude condition \eqref{KappaTightnessCondition1} with the same methods as in Proposition~\ref{TightnessN}.
To conclude tightness using the approach of Proposition~\ref{TightnessN} we need to verify that for all $t \in [0,T]$ and $\eta>0$
\begin{equation} \label{KappaLoss}
    \lim_{\delta \to 0} \limsup_{\kappa \to 0} \mathbb{P} (\mathcal{L}^\kappa_{t + \delta} - \mathcal{L}^\kappa_{t} \geq \eta)=0 
\end{equation}
where
\begin{equation*}
    \mathcal{L}_t^{\kappa} = \mathbb{P} \left(\tau^\kappa \leq t \vert \mathcal{F}_t^{\nu^\kappa,W^0}  \right)
\end{equation*}
to obtain tightness in the case $\kappa \to 0$ and the same condition with $\limsup_{\kappa \to \infty}$ for the other case. Markov's inequality yields
\begin{equation*}
    \mathbb{P} (\mathcal{L}^\kappa_{t + \delta} -  \mathcal{L}^\kappa_{t} \geq \eta) \leq \eta^{-1} \mathbb{P}(t<\tau^\kappa \leq t+ \delta) 
\end{equation*}
We consider a process $\tilde{Y}$ given by
\begin{equation*}
    \tilde{Y}_t^\kappa = X_0^\kappa + \int_0^t \mu_sds + \int_0^t \sigma_s dW_s
\end{equation*}
where $dW_s =  \rho_s dW_s^0 + \sqrt{1-\rho^2_s}dW_s^1$. The reflecting particle $X$ is then given by
\begin{equation*}
    X_t = \tilde{Y}_t + L_t
\end{equation*}
using the Skorokhod problem. Moreover, we set
\begin{equation*}
    Y_t^\kappa = \tilde{Y}_t^\kappa + \chi^\kappa.
\end{equation*}
We can apply this together with the definition of the stopping time $\tau^\kappa$ to obtain
\begin{align*}
    \mathbb{P} \left(\tau^\kappa \leq t \right) &= \mathbb{P} \left(L_t  \geq \chi^\kappa \right) = \mathbb{P} \left(\inf_{s \in [0,t]}\tilde{Y}_s \leq - \chi^\kappa \right) = \mathbb{P} \left(\inf_{s \in [0,t]}Y_s^\kappa \leq 0 \right).
\end{align*}
Taking $\varepsilon>0$ we have
\begin{align*}
    \mathbb{P} \left(t < \tau^\kappa \leq t + \delta \right) = & \mathbb{P} \left(t< \tau^\kappa \leq t + \delta, Y_t^\kappa \geq \varepsilon \right) + \mathbb{P} \left(t< \tau^\kappa \leq t+ \delta , Y_t^\kappa \in (0,\varepsilon) \right) \\
    &+ \mathbb{P} \left(t < \tau^\kappa \leq t + \delta, Y_t^\kappa \leq 0 \right).
\end{align*}
The last term is zero because $Y_t^\kappa \leq 0$ implies $\tau^\kappa \leq t$. This gives the estimate
\begin{equation} \label{CDFOfTau}
    \mathbb{P} \left(t < \tau^\kappa \leq t+\delta \right) \leq \mathbb{P} \left( t< \tau^\kappa \leq t + \delta, Y_t^\kappa \geq \varepsilon \right) + \mathbb{P} \left(Y_t^\kappa \in (0,\varepsilon) \right).
\end{equation}
To deal with the second term on the right-hand side we apply the scale transformation and change of measure discussed in Section \ref{probEstSection} to transform $Y^\kappa$ into a Brownian motion. The only dependence on $\kappa$ is then left in the initial condition. However, note that $\chi^\kappa$ converges to zero or infinity almost surely if we take the limit $\kappa \to \infty$ or $\kappa \to 0$, respectively. Since the function $\zeta$ is continuous in the $x$ variable and the measure $\mathbb{Q}$ equivalent to $\mathbb{P}$ we also get $\mathbb{Q}$-almost sure convergences to zero or infinity for $\zeta(t,\chi^\kappa)$. Using the properties of Brownian motion we then have
\begin{equation*}
    \limsup_{\kappa \to \infty} \mathbb{P} (Y_t^\kappa \in (0,\varepsilon)) \leq C_q \mathbb{Q} \left(B_t \in (0,C_{\mu,\sigma} \varepsilon) \right)^{1/q} = o(1), \quad \text{as $\varepsilon \to 0$},
\end{equation*}
for $q>1$, a constant $C_q>0$ depending on $q$ and $B$ a Brownian motion under $\mathbb{Q}$. In the case $\kappa \to 0$ we get
\begin{equation*}
    \limsup_{\kappa \to 0} \mathbb{P} (Y_t^\kappa \in (0,\varepsilon))=0. 
\end{equation*}
We follow the structure in the proof of \cite[Proposition 4.7 ]{HamblyLedger2017} to estimate the other term in \eqref{CDFOfTau}. We get
\begin{align*}
    \mathbb{P} \left(t < \tau^\kappa \leq t+ \delta, Y_t^\kappa \geq \varepsilon \right) & \leq \mathbb{P} \left(\inf_{s \in [t,t+\delta]} Y_s^\kappa \leq 0, Y_t^\kappa \geq \varepsilon \right) \\
    &\leq \mathbb{P} \left(\inf_{s \in [t,t+\delta]} (Y_s^\kappa-Y_t^\kappa) \leq - \varepsilon \right).
\end{align*}
Using the scale transformation $\zeta$ from Section \ref{probEstSection} we define
$U_s^\kappa \coloneqq \zeta(t+s,Y_{t+s}^\kappa)-\zeta(t+s,Y_t^\kappa)$. The dynamics of $U^\kappa$ are given by
\begin{equation*}
    dU_s^\kappa = u_s^\kappa ds + dW_s
\end{equation*}
where $u^{\kappa}$ is a uniformly bounded drift coefficient. This means we can find a constant $c_1>0$ such that
\begin{align*}
    \mathbb{P} \left(\inf_{s \in [t,t+\delta]} (Y_s^\kappa-Y_t^\kappa) \leq - \varepsilon \right)  &= \mathbb{P} \left(\inf_{s \in [0,\delta]} U_s^\kappa \leq - \kappa \right) \\
    &\leq \mathbb{P} \left(\inf_{s \in [0,\delta]}W_s \leq - c_1 (\varepsilon-\delta) \right) = \Phi(-c_1 \delta^{-1/2}(\varepsilon-\delta))
\end{align*}
where $\Phi$ is the normal c.d.f. We need that $\varepsilon-\delta>0$ and the convergences
\begin{equation*}
    \delta^{-1/2}(\varepsilon(\delta)-\delta) \to \infty, \quad \varepsilon(\delta) \to 0.
\end{equation*}
Choosing $\varepsilon(\delta) = \delta^{1/2}\log(1/\delta)$ gives the desired convergences and we get (\ref{KappaLoss}). This completes the tightness proof.
\end{proof}

To show weak convergence, it is now sufficient to prove that the limit points solve the respective SPDE for the reflecting and absorbing cases. To do this, we again rely on the martingale approach described in Section \ref{convergenceSection}. Consider the maps $M^{\phi}(\nu^{\kappa}), S^{\phi}(\nu^{\kappa})$ and $ C^{\phi}(\nu^{\kappa})$ defined in (\ref{martingaleComponents}). We need to show that, as we take the limit $\kappa \to 0/ \infty$, we obtain the corresponding quantities for the reflecting and absorbing cases, respectively. 

\begin{proposition} \label{CasesIntegralConvergence}
 Fix $t \leq T$ and define $\Psi_h: \{\xi \in D_{\mathcal{S}^{\prime}}: \xi_s \in \mathbf{M}_{\leq 1}(\R) \} \times \mathcal{S} \to \R$ by
\begin{equation}
    \Psi_h(\xi, \phi) \coloneqq \int_0^t \langle \xi_s, h(s,\cdot)\phi \rangle ds
\end{equation}
with $h$ denoting a placeholder for $\mu, \sigma^2$ or $\rho \sigma$. If  $\nu^{\kappa} \to \nu^*$ weakly on $(D_{\mathcal{S}^{\prime}},M1)$ and $\phi^\kappa \to \phi^*$ in $\mathcal{S}$, then we also have the convergence $\Psi^{\kappa}_h \to \Psi^*_h$ weakly on $\R$.
\end{proposition}

\begin{proof}
We show the result for the case $\kappa \to 0$. The other case follows using the same line of argument.
We fix a bounded Lipschitz function $f \in Lip(\R)$. By Lemma \ref{tightnessKappa} and the resulting relative compactness we can find a weakly convergent subsequence, again denoted by $(\nu^{\kappa}, \phi^{\kappa})$. By Skorokhod representation we can assume almost sure convergence and use the triangle inequality and the Lipschitz continuity of $f$ to obtain
\begin{align*}
    \left\lvert \mathbb{E} \left[f(\Psi^{\kappa}_h) \right] - \mathbb{E} \left[f(\Psi^{*}_h) \right] \right\rvert \leq & C \Bigg( \mathbb{E} \left[ \left\lvert \int_0^t \langle \nu_s^{\kappa} - \nu_s^*, h(s,\cdot) \phi^* \rangle ds  \right\rvert \right] \\
    & + \mathbb{E} \left[\left\lvert \int_0^t \langle \nu_s^{\kappa}, h(s,\cdot) \phi^* - h(s,\cdot) \phi^{\kappa} \rangle ds \right\rvert \right] \Bigg) \\
    \eqqcolon & C \left( \mathbb{E} [ I_1^{\kappa} ] + \mathbb{E} [ I_2^{\kappa}] \right)
\end{align*}
for some $C>0$. We can treat the first term as in \cite[Proposition 4.7]{HamblySojmark2019}. 
For the second term involving $I_2^{\kappa}$, using that $h$ is bounded and $\nu_s^{\kappa}$ is a sub-probability measure gives

\begin{equation*}
    I_2^{\kappa} \leq c \norm{\phi^*-\phi^{\kappa}}_{\infty} \to 0 \quad \text{as } \kappa \to 0.
\end{equation*}
This completes the proof for convergence to the reflecting case. The same arguments show the result for the case $\kappa \to \infty$.
\end{proof}

\begin{proof}[Proof of Theorem \ref{CaseConvergenceThm}]
Combining the martingale approach we used in Section \ref{convergenceSection} with the results of Lemma \ref{tightnessKappa} and Proposition \ref{CasesIntegralConvergence} it follows that $\nu^{\kappa}$ converges to a solution to the reflecting or absorbing SPDE when taking the limit $\kappa \to 0$ and $\kappa \to \infty$, respectively. The necessary regularity results follow in the same way as in Section \ref{regularitySection}. We can then conclude weak convergence from the uniqueness result of Theorem \ref{UniquenessReflecting} for the reflecting case and \cite[Theorem 2.6]{HamblySojmark2019}  for the absorbing case.
\end{proof}

\appendix
\section{Appendix: Properties of the Elastic Heat Kernel}

We begin by introducing several heat kernels, corresponding to a standard Brownian motion either on the whole real line or on the positive half-line with either reflection or elastic killing imposed at the origin.

Recall first that the Gaussian heat kernel on $\mathbb{R}$ with variance $\varepsilon$, which we denote by $p_{\varepsilon}$, is given by
\begin{equation} \label{GaussianHeatKernel}
    p_{\varepsilon}(x) = \frac{1}{\sqrt{2 \pi \varepsilon}} e^{- \frac{x^2}{2\varepsilon}}.
\end{equation}
Starting from this, the reflecting (or Neumann) heat kernel $G_{\varepsilon}^R$ on $[0,\infty)$ can then be defined as
\begin{equation} \label{reflectingKernel}
    G_{\varepsilon}^R(x,y) \coloneqq p_{\varepsilon}(x-y) + p_{\varepsilon}(x+y).
\end{equation}
Finally, the elastic (or Robin) heat kernel $G_{\varepsilon}^{E,\kappa}$ is given by \cite[Appendix 1, Sect.~10]{BorodinSalminen2002}.
\begin{align} \label{ElasticKernel}
\begin{split}
    G_{\varepsilon}^{E,\kappa}(x,y) =& \frac{1}{\sqrt{2 \pi \varepsilon}} \left[e^{-(x-y)^2/2 \varepsilon} + e^{-(x+y)^2/2 \varepsilon} \right]\\
    &-\kappa \exp \left(\kappa(x+y) + \frac{\kappa^2 \varepsilon}{2} \right) \left(1- \mbox{Erf}\left(\frac{x+y+\kappa \varepsilon}{\sqrt{2 \varepsilon}} \right) \right) \\
    =& G_{\varepsilon}^R(x,y) - g_{\varepsilon}^{E,\kappa}(x,y)
\end{split}    
\end{align}
where $\mbox{Erf}$ is the error function
\begin{equation*}
    \mbox{Erf}(x) = \frac{2}{\sqrt{\pi}} \int_0^x e^{-z^2}dz.
\end{equation*}

\begin{lemma}
The two maps $ x \mapsto G_{\varepsilon}^{E,\kappa}(x,y)$, for $y \geq 0$, and $y \mapsto G_{\varepsilon}^{E,\kappa}(x,y)$, for $x \geq 0$, satisfy the elastic boundary condition and belong to $\mathcal{C}^{E,\kappa}_0(\mathbb{R})$ for all $\varepsilon >0$.
\end{lemma}

\begin{proof}
This follows from direct computation of the derivatives.
\end{proof}

We need a relation between derivatives in $x$ and $y$ of $G_{\varepsilon}^{E,\kappa}$ to be able to switch derivatives from one variable to the other. This is the content of the following lemma.

\begin{lemma} \label{derivativeSwitch}
The derivatives of the elastic heat kernel $G_{\varepsilon}^{E,\kappa}$ satisfy
\begin{enumerate}
    \item $\partial_y G_{\varepsilon}^{E,\kappa}(x,y) = - \partial_x G_{\varepsilon}^{E,\kappa}(x,y) + 2 \partial_x p_{\varepsilon}(x+y) - 2 \partial_x g_{\varepsilon}^{E,\kappa}(x,y)$
    \item  $\partial_{yy}G_{\varepsilon}^{E,\kappa}(x,y) = \partial_{xx}G_{\varepsilon}^{E,\kappa}(x,y)$.
\end{enumerate}
\end{lemma}

\begin{proof}
(i) We have
\begin{align*}
    \partial_y G_{\varepsilon}^{E,\kappa}(x,y) &= \partial_y G_{\varepsilon}^R(x,y) - \partial_y g_{\varepsilon}^{E,\kappa}(x,y) \\
    \partial_x G_{\varepsilon}^{E,\kappa}(x,y) &= \partial_x G_{\varepsilon}^R(x,y) - \partial_x g_{\varepsilon}^{E,\kappa}(x,y).
\end{align*}
Note that by symmetry of $g_{\varepsilon}^{E,\kappa}$ in $x$ and $y$ we have $\partial_y g_{\varepsilon}^{E,\kappa}(x,y) = \partial_x g_{\varepsilon}^{E,\kappa}(x,y)$. Recall that the reflecting heat kernel is given by
\begin{equation*}
    G_{\varepsilon}^R(x,y) = p_{\varepsilon}(x-y) + p_{\varepsilon}(x+y).
\end{equation*}
Thus, we get the relation
\begin{equation*}
    \partial_y G_{\varepsilon}^R(x,y) = - \partial_x G_{\varepsilon}^R(x,y) + 2 \partial_x p_{\varepsilon}(x+y).
\end{equation*}
Considering the two  equations for the elastic heat kernel we have
\begin{equation*}
    \partial_y G_{\varepsilon}^{E,\kappa}(x,y) = - \partial_x G_{\varepsilon}^{E,\kappa}(x,y) + 2 \partial_x p_{\varepsilon}(x+y) - 2 \partial_x g_{\varepsilon}^{E,\kappa}(x,y).
\end{equation*}\\
(ii) First of all, it is easy to see that
\begin{equation*}
    \partial_{yy} G_{\varepsilon}^R(x,y) = \partial_{xx} G_{\varepsilon}^R(x,y).
\end{equation*}
The same relation holds for $g_{\varepsilon}^{E,\kappa}$ by symmetry in $x$ and $y$. This gives the result.
\end{proof}
We also need a bound on the elastic correction term $g_\varepsilon^{E,\kappa}$.
\begin{lemma}\label{ElasticCorrectionTermEstimate}
For all $x,y \geq 0$ and all $\varepsilon>0$ we have
  $  g^{E,\kappa}_{\varepsilon}(x,y) \leq \kappa \exp (-(x+y)^2/2\varepsilon)$.
\end{lemma}

\begin{proof}
By definition of the error function and a change of variable, we can estimate
\begin{align*}
    g_{\varepsilon}^{E,\kappa}(x,y) &= \kappa e^{\kappa (x+y) + \frac{\kappa^2 \varepsilon}{2}} \left(1-\mbox{Erf}\left(\frac{x+y+\kappa \varepsilon}{\sqrt{2 \varepsilon}} \right) \right) \\
    &= 2\kappa e^{\kappa (x+y) + \frac{\kappa^2 \varepsilon}{2}} \frac{1}{\sqrt{2 \pi \varepsilon}} \int_0^{\infty} e^{-\frac{(z+x+y+\kappa\varepsilon)^2}{2\varepsilon}}dz\\
    &=2 \kappa e^{\kappa(x+y)+\frac{\kappa^2 \varepsilon}{2}} \frac{1}{\sqrt{2 \pi \varepsilon}} \int_0^{\infty} e^{-\frac{(z+x+y)^2}{2\varepsilon}} e^{-\frac{\kappa^2\varepsilon}{2}} e^{-\kappa(z+x+y)}dz \\
    &\leq \kappa \frac{2}{\sqrt{2 \pi \varepsilon}} \int_0^{\infty} e^{-\frac{z^2}{2\varepsilon}} e^{-\frac{(x+y)^2}{2\varepsilon}} dz \leq \kappa e^{-\frac{(x+y)^2}{2\varepsilon}}.
\end{align*}
\end{proof}

\section{Appendix: Elastic Heat Kernel Mollification in $H^{-1}$} \label{AppendixH1}

Let $\zeta$ be a finite signed measure and let $p_{\varepsilon}$ denote the Gaussian heat kernel. The convolution $\zeta * p_{\varepsilon}$ given by
\begin{equation}
    \zeta * p_{\varepsilon}(x) = \int_{\R} p_{\varepsilon}(x-y) \zeta(dy)
\end{equation}
is a function in $C^{\infty}(\R)$ and the sequence $(\zeta * p_{\varepsilon})$ converges weakly to $\zeta$ as $\varepsilon \to 0$, i.e. for every bounded and continuous function $\phi: \R \to \R$ we have
\begin{equation*}
    \int_{\R} \phi(x) (\zeta * p_{\varepsilon})(x) dx \to \int_{\R} \phi(x) \zeta(dx), \quad \text{as } \varepsilon \to 0.
\end{equation*}
This smooth approximation is the basis of the kernel smoothing method but instead of the Gaussian heat kernel we use the elastic heat kernel $G_{\varepsilon}^{E,\kappa}$. We define the smoothed measure $T^{E,\kappa}_{\varepsilon} \zeta$ as
\begin{equation} \label{convDefn}
    T^{E,\kappa}_{\varepsilon} \zeta(x) = \langle \zeta, G_{\varepsilon}^{E,\kappa}(x,\cdot) \rangle = \int_0^{\infty} G_{\varepsilon}^{E,\kappa}(x,y) \zeta(dy).
\end{equation}
The reason for this choice is that it is an element of the test function space $\mathcal{C}^{E,\kappa}_0(\mathbb{R})$. This means we can use it in the weak formulation of the SPDE. Furthermore, it is known in an explicit form which enables explicit computations such as the switching of derivatives in Lemma \ref{derivativeSwitch}. The  weak convergence  still holds for $T^{E,\kappa}_{\varepsilon} \zeta$ in the case we are interested in. Furthermore, we have the following property.
\begin{proposition}[Contraction] \label{L2contraction}
If $f \in L^2(0,\infty)$, then $\Vert T^{E,\kappa}_{\varepsilon}f \Vert_2 \leq \Vert f\Vert_2$ for all $\varepsilon>0$.
\end{proposition}

\begin{proof}
Note that $G^{E,\kappa}_{\varepsilon} \geq 0$ and apply the Cauchy-Schwarz inequality to get
\begin{align*}
    \lvert T^{E,\kappa}_{\varepsilon} f(x) \rvert^2 &= \lvert \int_0^{\infty} G_{\varepsilon}^{E,\kappa}(x,y) f(y) dy \rvert^2 \\
    &= \lvert \int_0^{\infty} (G_{\varepsilon}^{E,\kappa}(x,y))^{\frac{1}{2}}(G_{\varepsilon}^{E,\kappa}(x,y))^{\frac{1}{2}} f(y) dy \rvert^2 \\
    & \leq \int_0^{\infty} G_{\varepsilon}^{E,\kappa}(x,y)dy \cdot \int_0^{\infty} G_{\varepsilon}^{E,\kappa}(x,y)f(y)^2dy.
\end{align*}
Integrating over $x \geq 0$ yields the result.
\end{proof}

The space in which we perform the energy estimates is $H^{-1}$, the dual space of $H^1$. The Sobolev space $H^1$ is the space of $L^2$-functions with weak derivative in $L^2$ equipped with the norm
\begin{equation*}
    \norm{f}_{H^1(0,\infty)} \coloneqq \left( \norm{f}^2_{L^2(0,\infty)} + \norm{\partial_x f}^2_{L^2(0,\infty)} \right)^{1/2}.
\end{equation*}
The space $H^{-1}$ is its dual space given by the linear functionals on $H^1$ with norm
\begin{equation*}
    \norm{\zeta}_{-1} \coloneqq \sup_{\norm{\phi}_{H^1(0,\infty)}=1} \lvert \zeta(\phi) \rvert.
\end{equation*}
The following proposition is adapted from \cite{HamblyLedger2017}. It justifies that the empirical measures and their limit points are indeed valued in $H^{-1}$.

\begin{proposition}
Let $\zeta$ be a finite signed measure. Then $\zeta \in H^{-1}$.
\end{proposition}

Another tool we need for the energy estimates is the notion of an anti-derivative operator $\partial_x^{-1}$, which we define as
\begin{equation} \label{antiDeriv}
    \partial_x^{-1} f(x) = - \int_x^{\infty} f(y)dy
\end{equation}
for an integrable function $f: \R \to \R$. Note that we then have $\partial_x \partial_x^{-1} f = f$. We use this operator in the following estimate for the $H^{-1}$-norm.

\begin{lemma}\label{H1lemma}
Let $\zeta  \in  H^{-1}$. Let furthermore the convolution operator $T^{E,\kappa}_{\varepsilon}$ and the anti-derivative $\partial_x^{-1}$ be defined as above. Then there exists a constant $C>0$ such that
\begin{equation} \label{H-1Estimate}
    \norm{\zeta}_{-1} \leq C \liminf_{\varepsilon \to 0} \left(\norm{\partial_x^{-1}T_{\varepsilon}^{E,\kappa}\zeta}_{L^2(0,\infty)}+ \left\lvert \int_0^{\infty}T_{\varepsilon}^{E,\kappa}\zeta(y)dy \right\rvert \right).
\end{equation}
\end{lemma}

\begin{proof}
Take a function $\phi \in H^1(0,\infty)$. By integration by parts we have
\begin{align*}
    \langle\zeta,T_{\varepsilon}^{E,\kappa} \phi \rangle = (T_{\varepsilon}^{E,\kappa} \zeta, \phi)_{L^2(0,\infty)} &=  \int_0^{\infty} \partial_x \partial_x^{-1} T^{E,\kappa}_{\varepsilon} \zeta(x) \phi(x) dx  \\
    &= - \partial_x^{-1} T_{\varepsilon}^{E,\kappa} \zeta(0) \phi(0) - \int_0^{\infty} \partial_x^{-1} T_{\varepsilon}^{E,\kappa} \zeta (x) \partial_x \phi(x) dx,
\end{align*}
where the boundary term at infinity vanishes due to the decay $\partial_x^{-1} T_{\varepsilon}^{E,\kappa} \zeta(x) \rightarrow 0$, as $x\rightarrow \infty$, and Morrey's inequality \cite{evans}, which gives $\norm{\phi}_{\infty} <\infty$. Furthermore, Morrey's inequality also gives $\phi(0)\leq C \norm{\phi}_{H^1}$, for a universal constant $C>0$, and hence
\begin{align*}
    \lvert \zeta(\phi) \rvert & \leq \liminf_{\varepsilon \to 0}\left(C \lvert \partial_x^{-1}T_{\varepsilon}^{E,\kappa}\zeta(0)\rvert \norm{\phi}_{H^1}  + \norm{\partial_x^{-1}T_{\varepsilon}^{E,\kappa}\zeta}_2 \norm{\phi}_{H^1}\right).
\end{align*}
Taking the supremum over all $\phi$ with $\norm{\phi}_{H^1}=1$ yields the result.
\end{proof}

\printbibliography

\noindent
This research has been supported by the EPSRC Centre for Doctoral Training in Mathematics of Random Systems: Analysis, Modelling and Simulation (EP/S023925/1).

\end{document}